\documentclass[11pt]{article}
\usepackage{epstopdf}
\usepackage{latexsym,amssymb,amsmath,amsfonts,amsthm}
\usepackage{mathrsfs,graphicx,graphics,subfigure}
\usepackage{placeins}
\usepackage{color,dsfont,enumerate}
\usepackage[vlined,ruled,linesnumbered,boxed]{algorithm2e}
\usepackage[font=small,labelfont=bf]{caption}
\usepackage{color}
\definecolor{hw}{rgb}{0,0,0}

\numberwithin{equation}{section}
\numberwithin{figure}{section}
\numberwithin{table}{section}
\newcommand{\R}{{\mathbb R}}

\newcommand{\Sp}{{\mathbb S}}

\newcommand{\ds}{\displaystyle}
\newcommand{\no}{\nonumber}
\newcommand{\be}{\begin{eqnarray}}
\newcommand{\ben}{\begin{eqnarray*}}
\newcommand{\en}{\end{eqnarray}}
\newcommand{\enn}{\end{eqnarray*}}
\newcommand{\ba}{\backslash}
\newcommand{\pa}{\partial}

\newcommand{\ov}{\overline}
\newcommand{\se}{\setminus}

\newcommand{\real}{{\rm Re\,}}

\newtheorem{theorem}{Theorem}[section]
\newtheorem{lemma}[theorem]{Lemma}

\newtheorem{remark}[theorem]{Remark}

\begin{document}
\renewcommand{\theequation}{\arabic{section}.\arabic{equation}}

\begin{titlepage}
\title{\bf %\mm{
Imaging of buried obstacles in a two-layered medium with phaseless far-field data}
%}
%
\author{Long Li\thanks{LMAM, School of Mathematical Sciences, Peking University, Beijing 100871,
China ({\tt 1601110046@pku.edu.cn})}
\and
{Jiansheng} Yang\thanks{LMAM, School of Mathematical Sciences, Peking University, Beijing 100871,
China ({\tt jsyang@pku.edu.cn})}
\and
Bo Zhang\thanks{NCMIS, LSEC and Academy of Mathematics and Systems Science, Chinese Academy of
Sciences, Beijing 100190, China and School of Mathematical Sciences, University of Chinese
Academy of Sciences, Beijing 100049, China ({\tt b.zhang@amt.ac.cn})}
\and
Haiwen Zhang\thanks{Corresponding author. NCMIS and Academy of Mathematics and Systems Science, Chinese Academy of Sciences,
Beijing 100190, China ({\tt zhanghaiwen@amss.ac.cn})}
}
\date{}
\end{titlepage}
\maketitle
%\vspace{.2in}

\begin{abstract}
The inverse problem we consider is to reconstruct the location and shape of buried obstacles in
the lower half-space of an unbounded two-layered medium in two dimensions from phaseless far-field data.
A main difficulty of this problem is that the translation invariance property of the modulus of
the far field pattern is unavoidable, which is similar to the homogenous background medium case.
Based on the idea of using superpositions of two plane waves with different directions as the incident
fields, we first develop a direct imaging method to locate the position of small anomalies and give
a theoretical analysis of the algorithm. Then a recursive Newton-type iteration algorithm in frequencies
is proposed to reconstruct extended obstacles.
Finally, numerical experiments are presented to illustrate the feasibility of our algorithms.

\vspace{.2in}
\end{abstract}

\section{Introduction and main results}\label{sec:introduction}

In this paper, we consider the inverse scattering by obstacles buried in a two-layered medium
separated by a flat plane and filled with different homogeneous materials,
which is essential to a broad spectrum of science and technology disciplines like geophysics,
underwater acoustics, and obstacle imaging in ocean environments.
For simplicity, we will focus our attention on the two-dimensional case.

Let $\R^2_{-}=\{{(x_1,x_2)}\in\R^2:x_2<0\}$ and $\R^2_{+}=\{{(x_1,x_2)}\in\R^2:x_2>0\}$ denote
the lower and upper half-spaces, respectively. The interface between the two layers is
denoted by $\Gamma=\{{(x_1,x_2)}\in\R^2:x_2=0\}$. We assume that the scattering obstacle $D$,
described by a bounded domain with {a connected complement}, is fully embedded in the
lower half-space $\R^{2}_{-}$.

Consider the incident wave $u^i=u^{i}(x,d,k_{+}):=e^{ik_{+}x\cdot d}$ propagating in the direction
\be\label{eq9}
d=(\cos\theta_d,\sin\theta_d), \quad \theta_d\in [\pi+\theta_c, 2\pi-\theta_c],
\en
where $\theta_c\in[0,\pi)$ is defined as
\ben
\theta_c :=\begin{cases}
\arccos({k_-}/{k_+}),&\quad k_{+}>k_{-},\\
0,  &\quad k_{+}<k_{-},\\
\end{cases}
\enn
with $k_\pm={\omega}/{c_\pm}>0$ being the wave numbers in $\R^2_\pm$, respectively.
Here,
$\omega$ is the wave frequency and $c_\pm$ are the wave speeds in the half-spaces $\R_\pm^2$,
respectively.
Further, the wave numbers $k_+$ and $k_-$ satisfy $k^2_-=nk^2_+$ with $n$ being the refractive
index.
The scattering problem in the two-layered medium is to find the total field $u= u^0+u^s$.
From the Fresnel formula, $u^0$ is given by
\ben
u^0(x,d,k_+,k_-)=\begin{cases}
\ds e^{ik_{+}x\cdot d}+ u^r(x,d,k_+,k_-),\quad & x\in\R^2_{+},\\
\ds u^t(x,d,k_+,k_-), \quad & x\in\R^2_{-},\\
\end{cases}
\enn
with
\ben
u^r(x,d,k_+,k_-):=R(\theta_d)e^{ik_{+}x\cdot d^r},\quad
u^t(x,d,k_+,k_-):=T(\theta_d)e^{ik_{-}x\cdot d^{t} },
\enn
where $d^{r}=(\cos\theta_d,-\sin\theta_d)$ is the reflection direction,
$d^{t}=(\cos\theta_d^t,\sin\theta_d^{t})$ is the transmission direction with
$\theta_d^{t}\in [\pi,2\pi]$ satisfying that $k_{+}\cos\theta_d=k_{-}\cos\theta_d^{t}$,
and the reflection and transmission coefficients $R(\theta_d)$ and $T(\theta_d)$ are given by
\be\label{eq:0}
{R(\theta_d)=\frac{k_{+}\sin\theta_d-k_{-}\sin\theta_d^t}{k_{+}\sin\theta_d+k_{-}{\sin\theta_d^t}},\;\;\;\;}
{T(\theta_d)=\frac{2k_{+}\sin\theta_d}{k_{+}\sin\theta_d+k_{-}\sin\theta_d^t}},
\en
respectively. The scattered wave $u^s$ produced by the interaction of $u^0$ with a sound-soft obstacle in presence
of the layered medium satisfies that
\be\label{eq:1}
\begin{cases}
\ds\Delta u^s +{k}^2u^{s}=0 \quad\quad\quad & \text{in}\quad\R^2\se(\ov{D}\cup\Gamma), \\
\ds[u^s]=0,\;\;\left[{\partial u^s}/{\partial\nu}\right]=0\quad\quad\quad &\text{on}\quad\Gamma,\\
\ds u^{s}=f\quad\quad\quad\quad\quad\quad \quad & \text{on}\quad\partial D,\\
{\ds
\lim_{r\rightarrow\infty}
\int_{\Sp^1_r}\left|\frac{\pa u^s}{\pa r}-ik u^s\right|^2ds=0},
&r=|x|,\quad x \in \R^2,\\
\end{cases}
\en
with $f=-{u^{0}}|_{\pa D}$,
where $\nu$ is the
unit normal vector on $\Gamma$ directed into $\R^2_+$,
$[\cdot]$ denotes the jump across the interface $\Gamma$,
$\Sp^1_r:=\{x \in\R^2 :|x|=r\}$ denotes the circle of radius $r$ centered at the origin
and the wave number $k$ is defined by
\ben
k=\begin{cases}
&k_{+}, \quad x\in\R^2_{+} ,\\
&k_{-}, \quad x\in\R^2_{-} .
\end{cases}
\enn
The well-posedness of the boundary value problem (\ref{eq:1}) can be obtained by employing a similar argument
as in \cite{CH98} where the case of
electromagnetic scattering has been considered.
In particular, it can be obtained
that the scattered wave $u^s$ has the asymptotic behavior \cite{AIL05}
\be\label{asymp}
u^{s}(x)=\frac{e^{ik_{+}|x|}}{\sqrt{|x|}}u^{\infty}(\hat x)
+o\left(\frac{1}{\sqrt{|x|}}\right),\quad |x|\rightarrow\infty,
\en
for all directions $\hat x={x}/{|x|}\in\Sp^1_+$, where $\Sp^1_+:=\{x = (x_1, x_2) :|x|=1,x_2 > 0\}$ denotes the upper unit half-circle
and $u^{\infty}(\hat x)$, defined on $\Sp^1_{+}$, is called the far-field pattern of the scattered wave $u^s$. {In the
present paper, given the incident wave $u^{i}(x,d,k_{+})=e^{ik_{+}x\cdot d}$, the corresponding scattered wave and far-field pattern are denoted by
$u^s(x,d,k_{+},k_{-})$ and $u^{\infty}(\hat x, d, k_{+},k_{-})$, respectively.}
{See Figure \ref{f1.1} for the problem geometry.}

\begin{figure}}
\centering
\includegraphics[width=0.65\textwidth]{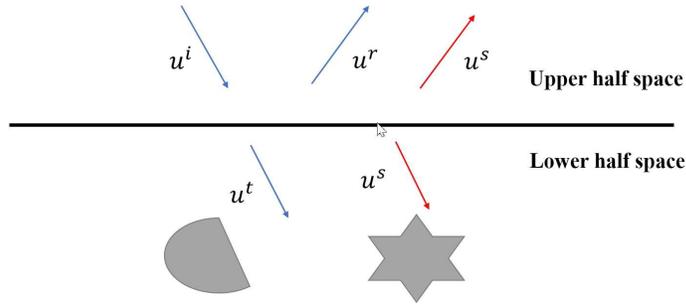}
\caption{The scattering problem by obstacles buried in a two-layered medium}{\label{f1.1}
\end{figure}
In this paper, we are concerned with the inverse scattering problem of recovering the buried obstacle
by wave detection made in the upper half-space.
This is a difficult problem due to the fact that the problem is both nonlinear and severely ill-posed,
which is a typical feature in inverse scattering problems. Various inversion algorithms have been developed
to tackle the nonlinearity and ill-posedness, being divided into two types of solution strategies:
iteration methods and non-iterative methods. Iteration methods usually make use of nonlinear constrained
optimization techniques with suitably choosing regularization terms to tackle the ill-posedness
(see, e.g., \cite{TH99,Hohage_1998,H98,Zhang_2013}).
Non-iterative methods usually deal with the nonlinearity property without using iteration.
Typical examples of non-iterative methods include the linear sampling method \cite{CK96,FD14},
the factorization method \cite{Kirsch_1998}, the method of topological derivatives \cite{A12,Bellis_2013}
as well as the MUSIC-type methods \cite{Cheney01,Kirsch02}. Recently, a new class of sampling methods
called direct imaging methods \cite{P10,Chen_2013,CHHA,Liu_2017} has been widely studied,
which is fast and highly robust to noises.

However, in many practical applications, the phase information of the far-field pattern is difficult and
sometimes impossible {to be measured}, and so only the intensity of the far-field pattern
(called the phaseless far-field data) is available.
Thus, in this paper, we restrict our attention to numerical methods for recovering the scattering obstacle $D$
embedded in the two-layered medium from the phaseless far-field data. {Inverse scattering with phaseless far-field pattern} is more
difficult than inverse scattering with full far-field data because of the translation invariance property
of the far-field pattern (see Lemma \ref{le:1} below) which makes it impossible to locate the position of the scattering obstacle.
Recently, the idea of using a superposition of two different plane waves to be the incident {field}
proposed in \cite{BLS} can effectively deal with the translation invariance of the far-field pattern.
Based on this idea, \cite{BLS} proposed a Newton method using multi-frequency measured data,
whilst \cite{RW} developed a direct imaging method to reconstruct the shape and location of the obstacle
without knowing the type of boundary conditions in a homogenous background medium.
Further, uniqueness can be guaranteed rigorously for recovering obstacles or inhomogeneous {media}
from the intensity of the far-field pattern generated by superpositions of two distinct plane waves
(see \cite{XuX2018} with certain conditions on the scatterers and \cite{XuX2018b} without
any condition on the scatterers but with a reference ball in the scattering system).
In addition, other solution strategies have been proposed to solve inverse scattering problems
with phaseless data numerically, such as an iteration method for only shape reconstruction \cite{I1,IK},
a direct imaging method based on a reference ball technique \cite{JLZ19a} and a direct imaging method
based on the reverse time migration technique \cite{Chen1,CFH17}.
For more results on phaseless inverse scattering problems including the uniqueness issue and other models,
see \cite{Dong19a,Dong19b,JLZ19b,JLZ19c,KM1,KM2,No2,No1,XuX2018,xzz19,G1}.

A large number of contributions exist to deal with the inverse problems in a two-layered background medium.
A MUSIC-type algorithm was first studied in \cite{AIL05} to determine the number and locations of small
inclusions buried in the lower half-space. An improved MUSIC-type method with multiple frequencies
was developed in \cite{Park_2010} to image thin inclusions. A direct imaging method was studied
in \cite{Li_2015} to recover multi-scale buried anomalies, which can locate small inclusions accurately
but needs a strong  {a priori} condition for recovering the shape of extended obstacles.
An asymptotic factorization method has been considered in \cite{G08} in the electromagnetic case.
These methods are robust to noises but require some {a priori} information on the small inclusions.
Other work on inverse scattering by extended obstacles can be found
in \cite{Coyle_2000,Gebauer_2005,K07,Delbary_2007}, where sampling-type methods and iteration-type methods
with phased near-field data have been studied for the layered-medium case.
The uniqueness issue was considered in \cite{LZ1} for the case of electromagnetic waves.

To the best of our knowledge, this paper is the first attempt to design an imaging algorithm with the
intensity of the far-field pattern to recover the location and shape of an obstacle in a two-layered
background medium. Since the translation invariance property is unavoidable, we follow the idea
of \cite{XuX2018,BLS,RW} and consider the incident wave $u^{i}=e^{ik_{+}x\cdot d_1}+e^{ik_{+}x\cdot d_2}$
with the incident directions $d_1,d_2$ satisfying (\ref{eq9}).
We first extend the direct imaging method in \cite{RW} from the homogenous background case to the
two-layered medium case to locate multiple small anomalies
with the intensity of the far-field pattern measured on the upper half-space.
A theoretical analysis of the imaging algorithm is then provided by using the {theory of oscillatory integrals}.
As shown by the results of numerical simulations, the direct imaging algorithm can accurately
and effectively determine the number and location of small scatterers.
However, the imaging algorithm gives poor shape reconstruction results for extended obstacles,
due to the limited aperture measurement data caused by refraction and reflection on the interface $\Gamma$.
Therefore, in order to obtain better results for the shape reconstruction of an extended obstacle embedded in the lower
half-space from the phaseless far-field data, we combine our direct imaging algorithm with
the recursive Newton iteration method developed in \cite{BLS}. Precisely,
the initial guess of the Newton iteration method is chosen with the help of the imaging results
obtained with the direct imaging algorithm. With such an initial guess, the recursive Newton
iteration method can satisfactorily and effectively recover the location and shape of the
extended obstacles, as illustrated by the numerical results.

The paper is organized as follows. In Section \ref{sec2}, we introduce the direct imaging method
for locating multiple small anomalies and give a theoretical analysis of this method
by the {theory of oscillatory integrals}.
In Section \ref{sec3}, a recursive Newton iteration method with multi-frequency phaseless far-field
data is developed to recover both the location and the shape of the extended obstacles buried in
the lower half-space. Section \ref{sec4} is devoted to the numerical experiments to illustrate
the good performance of the iterative method. {Finally, we will give some concluding remarks in Section \ref{sec5}}.

\section{Locating multiple small anomalies}\label{sec2}\setcounter{equation}{0}
In this section, we consider the inverse problem for determining the location of multiple small anomalies.
For this aim, we introduce the following notations.
Let $\Omega_j\subset\R^2_{-}$, $j=1,\ldots,Q$, be a family of base scatterers such that each $\Omega_j$
is simply connected and has a connected complement. Suppose that
$D=\cup^{Q}_{j=1}{D}_j\subset\R^2_{-}$
represents the multiple disjoint small scatterers, where $D_j=z_j+\rho{\Omega}_j,j=1,\ldots,Q$,
$\rho >0,\rho\ll 1$ and $L=\min_{1\le j,j^\prime\le Q, j\ne j^\prime}|z_j-z_{j^\prime}|\gg 1$.
Further, we also need the following notations for the rest of this paper.
For any $d$, $\hat{x}$ belonging to the unit circle $\Sp^1$,
let $d:=(\cos\theta_d,\sin\theta_d)$, $\hat{x}:=(\cos\theta_{\hat{x}},\sin\theta_{\hat{x}})$
with $\theta_d,\theta_{\hat{x}}\in[0,2\pi]$.
In particular,
for $d \in\Sp^{-}_{\theta_c}:=\{d=(\cos\theta_d,\sin\theta_d):\theta_d\in[\pi+\theta_c,2\pi-\theta_c]\}$, we define $d^t:= (\cos\theta_d^t,\sin \theta_d^t)$ and $T(\theta_d):= {2k_{+}\sin\theta_{d}}/{(k_+\sin\theta_{d}+k_-\sin\theta_{d}^t)}$, where $\theta_d$ and $\theta_d^t$ satisfy the relation
\be
\label{eq:2}
k_{+}\cos \theta_d = k_{-}\cos \theta_d^t, \quad {\theta^t_d}\in[\pi,2\pi].
\en
And for $\hat x \in\Sp^{+}_{\theta_c}:=\{\hat x =(\cos\theta_{\hat x},\sin\theta_{\hat x}):\theta_{\hat x}\in[\theta_c,\pi-\theta_c]\}$, we denote $\hat{x}^t:=(\cos\theta^t_{\hat x},\sin\theta^t_{\hat x})$ and $T(\theta_{\hat x}):= {2k_{+}\sin\theta_{\hat x}}/{(k_+\sin\theta_{\hat x}+k_-\sin\theta_{\hat x}^t)}$, where  $\theta^{t}_{\hat x}$ is defined by the relation
\be
\label{eq:3}
k_+\cos\theta_{\hat x}= k_-\cos\theta^t_{\hat x},\quad \theta^t_{\hat x}\in[0,\pi].
\en
Throughout this paper, the positive constants may be different
at different places.

We first stress that the translation invariance property of the phaseless far-field data is
inevitable in the two-layered background medium case, as stated by the following lemma.

\begin{lemma}\label{le:1}
Define $D_{z}:=D+z$ with $z=(z_1,0),\; z_1\in \R$.
For the incident wave $u^i(x,d,k_+)=\textcolor {red}{e^{ik_{+}x\cdot d}}$ with $d\in \Sp^{-}_{\theta_c}$,
the scattered waves $u^{\infty}(\cdot,d,k_+,k_-,D)$ and $u^{\infty}(\cdot,d,k_+,k_-,D_z)$ associated with
the obstacles $D$ and $D_z$, respectively, satisfy that
\be\label{eq:4.1}
u^{\infty}(\hat x,d,k_+,k_-,D_z)=e^{ik_{-}(z-\hat x)\cdot{{d}^t}}u^{\infty}(\hat x,d,k_+,k_-,D),\;\;\; \hat x\in \Sp^{+}_{\theta_c}.
\en
\end{lemma}
\begin{proof}
The proof of this Lemma is similar to that of Lemma \ref{le:1} in \cite{Li_2015}.
\end{proof}

By Lemma \ref{le:1}, we see that it is impossible to determine the location of the obstacle
using the modulus of the far-field pattern with only one incident plane wave. Therefore,
following the idea of \cite{RW}, we use the following superposition of two plane waves
as the incident field:
\ben
{\color{hw}{u^{i}(x,d_1,d_2,k_{+}):=u^{i}(x,d_1,k_{+})+u^{i}(x,d_2,k_{+})= e^{ik_{+}x\cdot d_1}+e^{ik_{+} x \cdot d_2}}}
\enn
with the incident directions $d_1,d_2\in\Sp^-_{\theta_c}$.
Then, by (\ref{asymp}) the corresponding scattered wave {$u^s(x,d_1,d_2,k_{+},k_{-})=u^{s}(x,d_1,k_{+})+u^{s}(x,d_2,k_{+})$} has the asymptotic behavior
\ben
u^s(x,d_1,d_2,k_{+},k_{-})=\frac{e^{ik_{+}|x|}}{\sqrt{|x|}}
u^{\infty}(\hat x,d_1,d_2,k_{+},k_{-})
+o\left(\frac{1}{\sqrt{|x|}}\right),\quad |x|\rightarrow\infty,
\enn
for all directions $\hat{x}=x/|x|\in\Sp^+_{\theta_c}$. By the linear superposition principle
it is clear that
\be\label{lsp}
u^{\infty}(\hat x,d_1,d_2,k_{+},k_{-})=u^{\infty}(\hat x,d_1,k_{+},k_{-})
+ u^{\infty}(\hat x,d_2,k_{+},k_{-}).
\en
The inverse problem considered in this paper is to recover the obstacle $D$ from the phaseless
far-field data $|u^{\infty}(\hat{x},d_1,d_2,k_+,k_-)|$ for $\hat x \in \Sp^{+}_{\theta_c}$ and $d_1,d_2\in\Sp^{-}_{\theta_c}$.

The purpose of this section is to present a direct imaging method with the phaseless far-field data
to solve the inverse problem numerically. The imaging function for continuous data is given by
\begin{align}\label{eq:4.4}
I(z,k_{+},k_{-})
=&\int_{\Sp^+_{\theta_c}}\int_{\Sp^-_{\theta_c}}\int_{\Sp^-_{\theta_c}}
|u^{\infty}(\hat x,d_1,d_2,k_+,k_-)|^2T(\theta_{d_1})e^{-ik_{-}z\cdot{d_1^t}}T(\theta_{d_2})
e^{ik_{-}z\cdot{d_2^t}}ds(d_1)ds(d_2)ds(\hat x)\no\\
&-\int_{\Sp^-_{\theta_c}}T(\theta_{d})e^{ik_{-}z\cdot{d^t}}ds(d)
\int_{\Sp^+_{\theta_c}}\int_{\Sp^-_{\theta_c}}{|u^{\infty}(\hat{x},d,k_+,k_-)|}^2T(\theta_{d})
e^{-ik_{-}z\cdot{d^t}}ds(d)ds(\hat x)\no\\
&-\int_{\Sp^-_{\theta_c}}T(\theta_{d})e^{-ik_{-}z\cdot {d^t}}ds(d)
\int_{\Sp^+_{\theta_c}}\int_{\Sp^-_{\theta_c}}{|u^{\infty}(\hat{x},d,k_+,k_-)|}^2T(\theta_{d})
e^{ik_{-}z\cdot{d^t}}ds(d)ds(\hat{x}),
\end{align}
where $d_j=(\cos\theta_{d_j},\sin\theta_{d_j})$ and $d_j^t=(\cos\theta_{d_j}^t,\sin\theta_{d_j}^t)$ with
$\theta_{d_j}^t$ and $\theta_{d_j}$ satisfying (\ref{eq:2}), $j=1,2$.
It follows from (\ref{lsp}) that $|u^{\infty}(\hat{x},d,k_+,k_-)|^2=|u^{\infty}(\hat{x},d,d,k_+,k_-)|^2/4$
for $\hat x \in \Sp^{+}_{\theta_c}$ and $d\in\Sp^{-}_{\theta_c}$.

We now study the behavior of $I(z,k_{+},k_{-})$ for locating multiple small anomalies.
We start with the following lemma concerning $I(z,k_{+},k_{-})$.

\begin{theorem}\label{th:1}
$I(z,k_{+},k_{-})=I_1(z,k_{+},k_{-})+I_2(z,k_{+},k_{-})$, $z\in\R^2,$ where
\ben
I_1(z,k_{+},k_{-})&=&\int_{\Sp^{+}_{\theta_c}}|v^{\infty}(\hat x,z,k_{+},k_{-})|^2ds(\hat x),\\
I_2(z,k_{+},k_{-})&=&\int_{\Sp^{+}_{\theta_c}}|w^{\infty}(\hat x,z,k_{+},k_{-})|^2ds(\hat x).
\enn
Here,
\ben
v^{\infty}(\hat x,z,k_+,k_-)&:=&\int_{\Sp^-_{\theta_c}}u^{\infty}(\hat x,d,k_+,k_-)
T(\theta_{{d}})e^{-ik_{-}z\cdot {{d}^t}}ds(d),\\
w^{\infty}(\hat x,z,k_+,k_-)&:=&\int_{\Sp^-_{\theta_c}}u^{\infty}(\hat x,d,k_+,k_-)
T(\theta_{{d}})e^{ik_{-}z\cdot{{d}^t}}ds(d)
\enn
are the far-field patterns of the scattering solutions to the scattering problem
\be\label{eq:4.6}
\begin{cases}
\ds\Delta u^s +{k}^2u^{s}=0 \quad\quad\quad & {\emph{in}}\quad\R^2\se(\ov{D}\cup\Gamma), \\
\ds[u^s]=0,\;\;\left[{\partial u^s}/{\partial\nu}\right]=0\quad\quad\quad &{\emph{on}}\quad\Gamma,\\
\ds u^{s}=f_{z}\quad\quad\quad\quad\quad\quad \quad & {\emph{on}}\quad\partial D,\\
{\ds
\lim_{r\rightarrow\infty}
\int_{\Sp^1_r}\left|\frac{\pa u^s}{\pa r}-ik u^s\right|^2ds=0},
&r=|x|,\quad x \in \R^2,
\\
\end{cases}
\en
with the boundary data
\be
\label{eq1}
&&f_z(x)=-\int_{\Sp^-_{\theta_c}}T^2(\theta_d)e^{ik_{-}(x-z)\cdot{d}^t}ds(d),\;\;x\in\pa D,\\
\label{eq2}
&&f_z(x)=-\int_{\Sp^-_{\theta_c}}T^2(\theta_{d})e^{ik_{-}(x+z)\cdot{d^t}}ds(d),\;\;x\in\pa D,
\en
respectively.
\end{theorem}

\begin{proof}
Inserting (\ref{lsp}) into (\ref{eq:4.4}) gives that
\begin{align*}
&I(z,k_{+},k_{-})\\
=&\int_{\Sp^{+}_{\theta_c}}\int_{\Sp^{-}_{\theta_c}}\int_{\Sp^{-}_{\theta_c}}
\Big[
u^{\infty}(\hat x,d_1,k_{+},k_{-})\ov{u^{\infty}(\hat x,d_2,k_{+},k_{-})}
+u^{\infty}(\hat x,d_2,k_{+},k_{-})\ov{u^{\infty}(\hat x,d_1,k_{+},k_{-})}
\\
&
+|u^{\infty}(\hat x,d_1,k_{+},k_{-})|^2
+|u^{\infty}(\hat x,d_2,k_{+},k_{-})|^2
\Big]
T(\theta_{d_1})e^{-ik_{-}z\cdot{d_1^t}}T(\theta_{d_2})e^{ik_{-}z\cdot{d_2^t}}ds(d_1)ds(d_2)ds(\hat x)\\
&-\int_{\Sp^{-}_{\theta_c}}T(\theta_d)e^{-ik_{-}z\cdot{d^t}}ds(d)\int_{\Sp^+_{\theta_c}}
\int_{\Sp^{-}_{\theta_c}}{|u^\infty(\hat{x},d,k_+,k_-)|}^2T(\theta_d)e^{ik_{-}z\cdot{d^t}}ds(d)ds(\hat{x})\\
&-\int_{\Sp^{-}_{\theta_c}}T(\theta_d)e^{ik_{-}z\cdot{d^t}}ds(d)\int_{\Sp^{+}_{\theta_c}}
\int_{\Sp^{-}_{\theta_c}}{|u^{\infty}(\hat x,d,k_+,k_-)|}^2T(\theta_d)e^{-ik_{-}z\cdot{d^t}}ds(d)ds(\hat{x}).
\end{align*}
Exchanging the order of integration, we have
\ben
I(z,k_{+},k_{-})&=&\int_{\Sp^{+}_{\theta_c}}\left|\int_{\Sp^{-}_{\theta_c}}
u^{\infty}(\hat x,d,k_+,k_-)T(\theta_d)e^{-ik_{-}z\cdot{d^t}}ds(d)\right|^2ds(\hat x)\\
&&+\int_{\Sp^{+}_{\theta_c}}\left|\int_{\Sp^{-}_{\theta_c}}u^{\infty}(\hat x,d,k_+,k_-)
T(\theta_d)e^{ik_{-}z\cdot{d^t}}ds(d)\right|^2ds(\hat x),
\enn
which is the required equality.
Since $u^{\infty}(\hat x, d,k_{+},k_{-})$ is the far-field pattern associated with the incident wave $u^{i}(x,d,k_{+})=e^{ik_{+}x\cdot d}$, it is easy to obtain that $v^{\infty}(\hat x,d,k_{+},k_{-})$
and $w^{\infty}(\hat x,d,k_{+},k_{-})$ are the far-field patterns of the scattering solutions to
the scattering problem (\ref{eq:4.6}) with the boundary data $f_z(x)$ given by (\ref{eq1}) and (\ref{eq2}), respectively.
The proof is complete.
\end{proof}

By Theorem \ref{th:1} we know that, in order to investigate the behavior of $I(z,k_{+},k_{-})$,
it is essential to know the property of the function
\be\label{eq:3.1}
B_0(y):=\int_{\Sp^{-}_{\theta_c}}T^2(\theta_{d})e^{ik_{-}y\cdot d^t}ds(d).
\en
In fact, it can be shown that $B_0(y)$ decays for $|y|$ large
enough. To this end, we need the following result in \cite{Chen1}, which is similar to Van der Corput's lemma for the oscillatory
integrals.

\begin{lemma}[see Lemma 3.9 in \cite{Chen1}]\label{le:2}
For any $-\infty< a<b<\infty$, let $u\in C^2[a,b]$ be real-valued and satisfy that $|u^\prime(t)|\ge 1$
for all $t\in[a,b]$. Assume that $a=x_0<x_1<\cdots<x_N=b$ is a partition of $[a,b]$ such that
$u^\prime$ is monotone in each interval $(x_{i-1},x_i)$, $i=1,2,\ldots,N$.
Then, for the smooth function $\psi$ defined on $(a,b)$ with integrable derivative
and for any $\lambda\ge 0$, we have
\ben
\int^b_a e^{i\lambda u(t)}\psi(t)dt
\le C(2N+2)\lambda^{-1}\left[|\psi(b)|+\int^b_a|\psi^\prime(t)|dt\right],
\enn
where $C$ is a positive constant independent of $\psi$ and $\lambda$.
\end{lemma}

With the aid of Lemma \ref{le:2}, we will prove the following lemma.

\begin{lemma}\label{le:3}
For $y\in\R^2$ with $|y|$ large enough, we have
\be\label{eq7}
|B_0(y)|\le C |y|^{-1/2},
\en
where $C>0$ is a constant independent of $y$.
\end{lemma}

\begin{proof}
Introducing the variable $\theta_d=\arccos[({k_-}/{k_+})\cos\theta_d^t]$ and noting that
\ben
d\theta_d=-\frac{k_-\sin\theta_d^t}{\sqrt{k_+^2-k_-^2(\cos\theta_d^t)^2}}d\theta_d^t,\;\;\;
\sin\theta_d=-(1/k_+)\sqrt{k_+^2-k_-^2(\cos\theta_d^t)^2},
\enn
we can rewrite $B_0(y)$ as
\be\label{eq:5}
B_0(y)&=&\int^{2\pi-\tilde{\theta}_c}_{\pi+\tilde{\theta}_c}f(\theta_d^t)
\exp\left[ik_-|y|(\cos\phi\cos\theta_d^t+\sin\phi\sin\theta_d^t)\right]d\theta_d^t\no\\
&=&\int^{2\pi-\tilde{\theta}_c}_{\pi+\tilde{\theta}_c}f(\theta_d^t)
\exp\left[ik_-|y|\cos(\phi-\theta_d^t)\right]d\theta_d^t,
\en
where
\ben
f(\theta_d^t)=\frac{-4k_{-}\sin{\theta_d^t}\sqrt{k_+^2-k_-^2(\cos\theta_d^t)^2}}
{\left[-\sqrt{k_+^2-k_-^2(\cos\theta_d^t)^2}+k_-\sin\theta_d^t\right]^2}.
\enn
Here, $\tilde{\theta}_c$ is defined as
\ben
\tilde{\theta}_c:=\begin{cases}
\ds 0,  &k_{-}<k_{+},\\
\ds \arccos({k_+}/{k_-})\in(0,\pi/2),&k_->k_+,\\
\end{cases}
\enn
and $y=|y|(\cos\phi,\sin\phi),\;\phi\in [0,2\pi]$.
It is easy to obtain that
\be\label{eq4}
\|f(\theta_d^t)\|_{C[\pi+\tilde{\theta}_c,2\pi-\tilde{\theta}_c]}
+\|f'(\theta_d^t)\|_{L^1[\pi+\tilde{\theta}_c,2\pi-\tilde{\theta}_c]}\leq C.
\en
Now the rest of the proof is split into two steps.

{\bf Step 1.} We first consider the case with $\phi\in[\pi,2\pi]$. Choose $\delta>0$ small enough such that
${2\delta}/\pi<\sin\delta$ and $0<\delta<({\pi-2\tilde{\theta}_c})/6$.
We distinguish between the following two cases.

{\bf Case 1:} $\phi\in[\pi,\pi+\tilde{\theta}_c+2\delta)\cup(2\pi-\tilde{\theta}_c-2\delta,2\pi]$. From the choice of $\delta$, we have $\pi+\tilde{\theta}_c+3\delta<2\pi-\tilde{\theta}_c-3\delta$
and thus split (\ref{eq:5}) into three parts:
\ben
B_0(y)=\left[\int^{\pi+\tilde{\theta}_c+3\delta}_{\pi+\tilde{\theta}_c}
+\int^{2\pi-\tilde{\theta}_c-3\delta}_{\pi+\tilde{\theta}_c+3\delta}+\int^{2\pi-\tilde{\theta}_c}_{2\pi-\tilde{\theta}_c-3\delta}\right]f(\theta_d^t)
\exp\left[ik_-|y|\cos(\phi-\theta_d^t)\right]d\theta_d^t
=:I_1+I_2+I_3.
\enn
Set $u(\theta_d^t):=({\pi}/2)\cos(\phi-\theta_d^t)/\delta$. Then $|u^\prime(\theta_d^t)|=|({\pi}/2)\sin(\phi-\theta_d^{t})/\delta|\ge|({\pi}/2)(\sin\delta)/\delta|\ge1$,
and $u^\prime(\theta_d^t)$ is piecewise monotone in ${[\pi+\tilde{\theta}_c+3\delta, 2\pi-\tilde{\theta}_c-3\delta]}$.
Thus by Lemma \ref{le:2}, we have
\be\label{eq5}
|I_2|\le \frac{C}{\delta |y|}\left(\|f(\theta_d^t)\|_{C[\pi+\tilde{\theta}_c,2\pi-\tilde{\theta}_c]}
+\|f'(\theta_d^t)\|_{L^1[\pi+\tilde{\theta}_c,2\pi-\tilde{\theta}_c]}
\right).
\en
It is easy to show that
\be\label{eq6}
|I_j|\le C{\delta}{\|f(\theta_d^t)\|}_{C[\pi+\tilde{\theta}_c,2\pi-\tilde{\theta}_c]},\;\;\; j=1,3.
\en
Combining the estimates (\ref{eq4}), (\ref{eq5}) and (\ref{eq6}) gives
\begin{align}\label{eq:4.3}
|B_0(y)|\le C\left(\delta+\frac 1{\delta |y|}\right)
\left(\|f(\theta_d^t)\|_{C[\pi+\tilde{\theta}_c,2\pi-\tilde{\theta}_c]}
+\|f'(\theta_d^t)\|_{L^1[\pi+\tilde{\theta}_c,2\pi-\tilde{\theta}_c]}
\right)
\le \widetilde{C}\left(\delta+\frac 1{\delta |y|}\right).
\end{align}
Therefore, taking $\delta=|y|^{-1/2}$ in (\ref{eq:4.3}) yields the estimate
(\ref{eq7}).

{\bf Case 2:} $\phi\in[\pi+\tilde{\theta}_c+2\delta,2\pi-\tilde{\theta}_c-2{\delta}]$.
We use a similar idea as in the proof of Case 1 and split (\ref{eq:5}) into three parts:
\ben
B_0(y)=\left[\int^{\phi-\delta}_{\pi+\tilde{\theta}_c} +\int^{\phi+\delta}_{\phi-\delta}
+\int^{2\pi-\tilde{\theta}_c}_{\phi+\delta}\right]f(\theta_d^t)
\exp\left[ik_-|y|\cos(\phi-\theta_d^t)\right]d\theta_d^t=:II_1+II_2+II_3.
\enn
Similarly as in the estimate of $I_2$ in Case 1, it is deduced that
\ben
|II_j|\le \frac{C}{\delta |y|}
\left(\|f(\theta_d^t)\|_{C[\pi+\tilde{\theta}_c,2\pi-\tilde{\theta}_c]}
+\|f'(\theta_d^t)\|_{L^1[\pi+\tilde{\theta}_c,2\pi-\tilde{\theta}_c]}
\right),
\quad j=1,3.
\enn
It is easy to see that $\ds |II_2|\le C\delta{\|f(\theta_d^t)\|}_{C[\pi+\tilde{\theta}_c,2\pi-\tilde{\theta}_c]}$.
Then similarly as in the proof of Case 1, we can obtain the estimate (\ref{eq7}).

{\bf Step 2.} Consider the case $\phi\in[0,\pi]$. Introduce the new variable $\psi=\phi+\pi$.
Then $\psi\in[\pi,2\pi]$ and
\ben
B_0(y)=\int^{2\pi-\tilde{\theta}_c}_{\pi+\tilde{\theta}_c}f(\theta_d^t)\exp\left[{-ik_-|y|\cos(\psi-\theta_d^t)}\right]d\theta_d^t.
\enn
Consequently, the rest proof of this case is similar to the arguments in Step 1.
The proof is thus complete.
\end{proof}

We now study the behavior of the imaging function $I(z,k_{+},k_{-})$.
To this end,
denote by $G(x,y)$ the fundamental solution of the unperturbed problem (\ref{eq:1}) with $D=\emptyset$,
which can be derived by the Fourier transform technique (see, e.g., \cite{Li2010}).
Define the single- and double-layer potentials
\begin{align*}
(\mathcal{S}\psi)(x)&=\int_{\pa D}G(x,y)\psi(y)ds(y),\quad x\in\R^2\ba\pa D,\\
(\mathcal{D}\psi)(x)&=\int_{\pa D}\frac{\pa G(x,y)}{\pa\nu(y)}\psi(y)ds(y),\quad x\in\R^2\ba\pa D,
\end{align*}
and the boundary integral operators
\begin{align*}
(S\psi)(x)&=\int_{\partial D}G(x,y)\psi(y)ds(y),\quad x\in\partial D,\\
(K\psi)(x)&=\int_{\pa D}\frac{\pa G(x,y)}{\pa\nu(y)}\psi(y)ds(y),\quad x\in\partial D.
\end{align*}
It is well known that $G(x,y)$ has the asymptotic formula \cite{AIL05}:
\begin{align}\label{eq3}
G(x,y)=\frac{e^{i\pi/4}}{\sqrt{8\pi k_+}}\frac{e^{ik_{+}|x|}}{\sqrt{|x|}}
T(\theta_{\hat x})e^{-ik_- y\cdot\hat{x}^t}+
o\left(\frac{1}{\sqrt{|x|}}\right),\quad|x|\rightarrow\infty,
\end{align}
where $\hat x=x/|x|=(\cos\theta_{\hat x},\sin\theta_{\hat x})\in \Sp^{+}_{\theta_c}$
and
$\hat{x}^t=(\cos\theta^t_{\hat x},\sin\theta^t_{\hat x})$ with
$\theta^{t}_{\hat x}$ and $\theta_{\hat x}$ satisfying (\ref{eq:3}).

Define
\begin{align*}
(S^{\infty}\psi)(\hat x)&=\frac{e^{i\pi/4}}{\sqrt{8\pi k_+}}
\int_{\pa D}T(\theta_{\hat x})e^{-ik_{-} y\cdot\hat{x}^t}\psi(y)ds(y),\quad\hat{x}\in\Sp^{+}_{\theta_c},\\
(K^{\infty}\psi)(\hat x)&=\frac{e^{i\pi/4}}{\sqrt{8\pi k_+}}\int_{\pa D}
\frac{\pa T(\theta_{\hat x})e^{-ik_{-}y\cdot\hat{x}^t }\psi(y)}{\pa\nu(y)}ds(y),\quad\hat x\in\Sp^+_{\theta_c}.
\end{align*}
From (\ref{eq3}), it is clear that $(S^{\infty}\psi)(\hat x)$ and $(K^{\infty}\psi)(\hat x)$ are the far-field patterns
on $\Sp^+_{\theta_c}$ of $(S\psi)(x)$ and $(D\psi)(x)$, respectively.

From Theorem {\ref{th:1}} it follows that
\ben
v^s(x,z,k_+,k_-)=\int_{\Sp^-_{\theta_c}}u^s(x,d,k_+,k_-)T(\theta_{{d}})e^{-ik_{-} z\cdot{d}^t}ds(d)
\enn
is the scattering solution to the problem (\ref{eq:4.6}) with boundary data $f_z$ given by (\ref{eq1}).
From \cite{Li2010} it is known that $G(x,y)=\Phi(x,y)+H(x,y)$ for $x,y\in\R^2_{-}$ and $x\neq y$, where
\ben
\Phi(x,y)=\frac{i}{4} H^1_0(k_{-}|x-y|),\;\;\;x\not=y,
\enn
is the fundamental solution of the Helmholtz equation $\Delta w +{k}^2_-w=0$ in $\R^2$ with $H^1_0$ being the Hankel
function of the first kind of order zero and
$H(x,y)\in C^{\infty}(\R^2_{-}\times\R^2_{-})$ accounts for
the {reflection} due to the layered medium. Thus it is easy to derive that $S$ and $K$ are compact
perturbations of the corresponding integral operators associated with the homogeneous problem
(i.e., with $G$ replaced by $\Phi$). Therefore, by using a similar argument as in the proof of
Theorem 3.11 in \cite{DK13}, we can seek the solution $v^s(x,z,k_{+},k_{-})$ in the form of
combined double- and single-layer potential with density $\phi_z\in C(\pa D)$, that is,
\be\label{eq:4.5}
v^{s}(x,z,k_{+},k_{-})=(\mathcal{D}\phi_z)(x)-i(\mathcal{S}\phi_z)(x),
\quad x\in\R^2\ba\ov{D},
\en
where $\phi_z$ is the unique solution to the boundary integral equation
\ben
A\phi_z:={\left(\frac 12 I+ K\right)\phi_z- iS\phi_z}=f_z
\enn
with $f_z$ given by (\ref{eq1}).
Arguing similarly as in the proof of Theorem 3.11 in \cite{DK13}, we can prove that the operator $A$
is bijective and invertible in $C(\pa D)$. Thus we have
\be\label{eq:4.7}
C_1\|f_z\|_{C(\pa D)}\le\|\phi_z\|_{C(\pa D)}\le C_2\|f_z\|_{C(\pa D)}
\en
with two positive constants $C_1$ and $C_2$ independent of $z$.

On the other hand, we have, by Lemma \ref{le:3}, that for $x\in\pa D$,
\ben
f_z(x)=\begin{cases}
\ds-\int_{\Sp^{-}_{\theta_c}}{T^2(\theta_{{d}})}ds(d), &\text{if}\;z=x,\\
\ds O(|x-z|^{-1/2}), \quad & \text{if}\;|z-x|\gg1.
\end{cases}.
\enn
Let $d(z,\pa D)$ be the distance between $z$ and $\pa D$. Then it follows from (\ref{eq:4.7}) that
\be\label{eq:4.8}
\begin{cases}
\ds\|\phi_z\|_{C(\pa D)}\ge C_1\int_{\Sp^{-}_{\theta_c}}{{T^2(\theta_{{d}})}}ds(d),&\text{if}\;z\in\pa D,\\
\ds\|\phi_z\|_{C(\pa D)}= O(d(z,\pa D)^{-1/2}), & \text{if}\;d(z,\pa D)\gg 1.
\end{cases}
\en
Since $v^{\infty}(\hat x,z,k_+,k_-)$ is the far-field pattern of the scattered field $v^s(x,z,k_{+},k_{-})$,
we have
\ben
v^{\infty}(\hat x,z,k_{+},k_{-})=(K^{\infty}\phi_z)(\hat x) - i(S^{\infty}\phi_z)(\hat x),
\quad\hat x\in\Sp^{+}_{\theta_c}.
\enn
Note that
\ben
I_1(z,k_{+},k_{-})=\int_{\Sp^{+}_{\theta_c}}|v^{\infty}(\hat x,z,k_{+},k_{-})|^2 ds(\hat x).
\enn
Then, by (\ref{eq:4.8}) and the properties of the operators $S^{\infty}$ and $K^{\infty}$,
it can be seen that $I_1(z,k_{+},k_{-})$
decays as $z$ moves away from $D$.

Similarly as for the analysis of $I_1(z,k_{+},k_{-})$, it can be seen that
\ben
I_2(z,k_{+},k_{-})= \int_{\Sp^{+}_{\theta_c}}|w^{\infty}(\hat x,z,k_{+},k_{-})|^2ds(\hat x)
\enn
decays as $z$ moves away from $D^\prime$, where $D^\prime$ is the symmetric obstacle
of $D$ with respect to the origin.

From what has been discussed above, it can be seen that the imaging functional $I(z,k_+,k_-)$ decays
as $z$ moves away from $D\cup D^\prime$. Further, based on the above analysis, it is reasonable
to expect that $I(z,k_+,k_-)$ will take a large value
in the neighborhood of $\pa D\cup\pa D^\prime$.
This is confirmed by numerical examples in Section \ref{sec4} though a rigorous analysis is not available yet.
According to the performance of $I(z,k_+,k_-)$
and the fact that $D'\subset\R^2_+$, it is enough to determine the location of the obstacle $D$
even though the obstacle $D$ is completely buried in the lower half-space.

\begin{remark}\label{re2.5} {\rm
The performance of $I(z,k_+,k_-)$ can also be studied by using the analysis in \cite{Li_2015}. In fact,
with the aid of Theorem 2.1 in \cite{Li_2015} and the smallness assumption on the obstacle $D$, we can easily obtain the following asymptotic formula:
\be\label{eq:2.6}
|v^{\infty}(\hat x, z,k_{+},k_{-})|=\frac{T(\theta_{\hat{x}})}{\ln\rho}
\left|\sum^{Q}_{j=1}c_j\int_{\Sp^{-}_{\theta_c}}{\color{hw}{e^{ik_{-}(z_j-z)\cdot d^{t}}}}T(\theta_d)ds(d)\right|
+O\left(\frac 1{(\ln\rho)^2}+\frac1{\sqrt L}\right)
\en
for sufficiently large $L$, as $\rho\rightarrow +0$, where $v^{\infty}(\hat x,z,k_{+},k_{-})$ is given in Theorem \ref{th:1}
and $c_j,\;j=1,\ldots,Q$, are constants depending on $\Omega_j,k_-,d$,
but independent of $\rho$. Similar to the proof of Lemma \ref{le:3}, it is easy to show that
\be\label{eq2.7}
\left|\int_{\Sp^-_{\theta_c}}{\color{hw}{e^{ik_{-}(z_j-z)\cdot d^{t}}}}T(\theta_d)ds(d)\right|\le C{|z_j-z|}^{-1/2},
\en
where $C$ is a constant independent of $z_j$ and $z$.
By (\ref{eq:2.6}) and (\ref{eq2.7}), $z_j$, $j=1,\ldots,Q$, can be seen as a local maximizer in a neighborhood
of $z_j$ of $I_1(z,k_{+},k_{-})$ defined in Theorem \ref{th:1}.
Similarly to the above discussion, $z_j'$, $j=1,\ldots,Q$, can be seen as a local maximizer in a neighborhood
of $z_j'$ of $I_2(z,k_{+},k_{-})$ defined in Theorem \ref{th:1},
where $z_j'$ is the symmetric point of $z_j$ with respect to the origin.
Therefore, it is expected that multiple small scatterers can be determined by using the imaging functional
$I(z,k_+,k_-)$.
This is in accordance with our analysis on $I(z,k_{+},k_{-})$.
}
\end{remark}

We are now  ready to give the direct imaging algorithm for the inverse problem.
Suppose that there are $n_f$ measurement points $\hat{x}_j\in \Sp^{+}_{\theta_c}$ $(j=1,2,\ldots,n_f)$
and $n^{(1)}_d$ sets of two incident directions $d^{(1)}_{1l},
d^{(1)}_{2i}\in\Sp^{-}_{\theta_c}$ $(l,i=1,2,\ldots, n^{(1)}_d)$,
{where $\hat x_j=(\cos\theta_{\hat x_j},\sin \theta_{\hat x_j})$ with $\theta_{\hat x_j}=\theta_c+{(j-1)({\pi-2\theta_c})}/{n_f}$, $d^{(1)}_{1l}=(\cos{\theta_{d^{(1)}_{1l}}},\sin\theta_{d^{(1)}_{1l}})$
with $\theta_{d^{(1)}_{1l}}=\pi+\theta_c+{(l-1)({\pi-2\theta_c})}/{n^{(1)}_d}$ and $d^{(1)}_{2i}=(\cos{\theta_{d^{(1)}_{2i}}},\sin\theta_{d^{(1)}_{2i}})$
with $\theta_{d^{(1)}_{2i}}=\pi+\theta_c+{(i-1)({\pi-2\theta_c})}/{n^{(1)}_d}$.} Let $P=\pi-2\theta_c$.
Then with the aid of the trapezoid quadrature rule, the continuous imaging function $I(z,k_+,k_-)$ given in
(\ref{eq:4.4}) can be approximated by the discrete imaging function $I_A(z,k_+,k_-)$ defined by
\begin{align}\label{eq:2.5}
&I_A(z,k_{+},k_{-})\no\\
=&\frac{P}{n_f}\left(\frac{P}{n^{(1)}_d}\right)^2\sum^{n_f}_{j=1}\sum^{n^{(1)}_d}_{l=1}\sum^{n^{(1)}_d}_{{\color{hw}i=1}}
\left|u^{\infty}(\hat x_j,d^{(1)}_{1l},d^{(1)}_{2i},k_+,k_-)\right|^2 T(\theta_{d^{(1)}_{1l}})
e^{-ik_{-}z\cdot{{(d^{(1)}_{1l})}^t}}T(\theta_{d^{(1)}_{2i}})e^{ik_{-}z\cdot{{(d^{(1)}_{2i})}^t}}\no\\
-&\frac{P}{n_f}\frac{P}{4n^{(1)}_d}\sum^{{n^{(1)}_d}}_{i=1}T(\theta_{d^{(1)}_{2i}})e^{ik_{-}z\cdot{{(d^{(1)}_{2i})}^t}}
\sum^{n_f}_{j=1}\sum^{n^{(1)}_d}_{l=1}\left|u^{\infty}(\hat x_j,d^{(1)}_{1l},d^{(1)}_{2l},k_{+},k_{-})\right|^2
T(\theta_{d^{(1)}_{1l}})e^{-ik_{-}z\cdot{(d^{(1)}_{1l})^t}}\no\\
-&\frac{P}{n_f}\frac{P}{4n^{(1)}_d}\sum^{n^{(1)}_d}_{l=1}T(\theta_{d^{(1)}_{1l}})e^{-ik_{-}z\cdot{{(d^{(1)}_{1l})}^t}}
\sum^{n_f}_{j=1}\sum^{n^{(1)}_d}_{i=1}\left|u^{\infty}(\hat x_j,d^{(1)}_{1i},d^{(1)}_{2i},k_+,k_-)\right|^2
T(\theta_{d^{(1)}_{2i}})e^{ik_{-}z\cdot{(d^{(1)}_{2i})^t}},
\end{align}
where we have employed the facts that $d^{(1)}_{1l}=d^{(1)}_{2l}$ $(l=1,\ldots, n^{(1)}_d)$
and $u^\infty(\hat{x},d,d,k_+,k_-)=2u^\infty(\hat{x},d,k_+,k_-)$
for $\hat x \in \Sp^{+}_{\theta_c}$ and $d\in\Sp^{-}_{\theta_c}$.
In the numerical experiments, we will consider the noisy phaseless far-field pattern $|u_{\delta}^{\infty}(\hat x_j,d^{(1)}_{1l},d^{(1)}_{2i},k_+,k_-)|,\;j=1,\ldots,n_f,\;
l=1,\ldots,n^{(1)}_d,\;i=1,\ldots,n^{(1)}_d$, as the measured data, where $|u_{\delta}^{\infty}(\hat x_j,d^{(1)}_{1l},d^{(1)}_{2i},k_+,k_-)|$ is the small perturbation of the phaseless far-field pattern $|u^{\infty}(\hat x_j,d^{(1)}_{1l},d^{(1)}_{2i},k_+,k_-)|$ with noise level $\delta>0$ (see the formula (\ref{eq6.1}) below).
Accordingly, the discrete imaging function
$I_A(z,k_+,k_-)$ with noisy phaseless far-field data can be computed by (\ref{eq:2.5}) with $|u^{\infty}(\hat x_j,d^{(1)}_{1l},d^{(1)}_{2i},k_+,k_-)|$ replaced by
$|u_{\delta}^{\infty}(\hat x_j,d^{(1)}_{1l},d^{(1)}_{2i},k_+,k_-)|$.
Finally, our direct imaging algorithm is based on the discrete imaging function
$I_A(z,k_+,k_-)$ and presented in Algorithm \ref{A1}.

\begin{algorithm}
\caption{\textbf{Locating multiple small anomalies}\label{A1}}
\SetAlgoLined
\KwIn{Noisy phaseless data $|u_{\delta}^{\infty}(\hat x_j,d^{(1)}_{1l},d^{(1)}_{2i},k_+,k_-)|,j=1,\ldots,n_f,\;
l=1,\ldots,n^{(1)}_d,\;i=1,\ldots,n^{(1)}_d$.}
\KwOut{The number and location of the small scatterers.}

Choose a sampling domain $\Omega_P\subset\R^2_{-}$ containing the obstacle $D$ with a mesh $\mathcal{T}$.

Compute the imaging function $I_A(z,k_{+},k_{-})$ with noisy phaseless far-field data for $z\in \mathcal{T}$.

Locate all the sampling points on $\mathcal{T}$ at which $I_A(z,k_{+},k_{-})$ takes a large value.
\end{algorithm}

\begin{remark}\label{re:2} {\rm
Since we have the a priori information that the scatterers are embedded in the lower half-space,
it is reasonable to choose the sampling region $\Omega_P\subset\R^2_-$.
This, combined with the property of $I(z,k_+,k_-)$
and the fact that $D'\cap\R^2_-=\varnothing$, makes it possible to determine the location
of the multiple anomalies of the obstacle $D$.
}
\end{remark}

\begin{remark}\label{re:1} {\rm
Algorithm \ref{A1} can be applied to determine the location of extended obstacle $D$. In fact,
by the same analysis discussed above, it is expected that the imaging function $I(z,k_+,k_-)$
decays as $z$ moves away from $D\cup D^\prime$ and reaches the  local maximums at some points in the
neighborhood of $\pa D\cup\pa D^\prime$.
The latter is not yet rigourously proved but will be confirmed by numerical examples in Section \ref{sec4}.
With the aid of the {a priori} information that $D$ is embedded in the lower half-space,
$I(z,k_{+},k_{-})$ is able to help us to determine the location of $D$ roughly, which will provide
our Newton-type iteration method
presented in next section with the initial guess.
}
\end{remark}

\section{Recovering the location and shape of extended obstacles}\label{sec3}
\setcounter{equation}{0}

As discussed in Remark \ref{re:1}, our direct imaging algorithm can determine the location of
extended obstacles roughly, providing some a priori information for the iteration-type method.
With the aid of the {\color{hw}{a priori}} information, we develop a recursive Newton-type iteration algorithm
in frequencies to recover the location and shape of extended obstacles embedded in the lower half-space.

Our aim is to solve the nonlinear and ill-posed equation
\be\label{eq:5.1}
F_{d_1,d_2,k_{+},k_{-}}[\pa D](\hat x)=|u^{\infty}(\hat x,d_1,d_2,k_{+},k_{-})|^2,\quad \hat{x}\in \Sp^+_{\theta_c},
\en
where the far-field operator $F_{d_1,d_2,k_{+},k_{-}}$ maps the boundary of the extended scatterer $\pa D$
to the corresponding phaseless far-field data induced by the incident wave
$u^{i}(x,d_1,d_2,k_+)=e^{ik_+x\cdot d_1}+e^{ik_+x\cdot d_2}$.
For simplicity, we assume that $D$ is simply connected in what follows.
For the case when the obstacle $D$ has
several connected components, see the discussion
in Remark \ref{re3.3}.

To proceed further, we need to characterize the Fr\'{e}chet derivative of the far field operator.
Similarly to \cite{BLS}, we choose the Hilbert space $L^2(\Sp^+_{\theta_c})$ of square integrable
functions on $\Sp^+_{\theta_c}$ as the data space, which is suitable for describing
the measurement error. Let $h\in C^2(\pa D)$ be a twice continuous differentiable vector field,
and define $\pa D_h:=\{y\in\R^2|y=x+h(x),x\in\pa D\}$. For a sufficient small $\ell>0$
depending on $\pa D$, each $\pa D_h$ with $\|h(x)\|_{C^2(\pa D)}\le \ell$ is also a $C^2$-smooth boundary
of a domain $D_h$. Denote the set $V_\ell:=\{h\in C^2{(\pa D)},\|h(x)\|_{C^2(\pa D)}\le \ell\}$,
then the far field operator is called Fr\'{e}chet differentiable at $\pa D$ if there exists a linear
mapping $F^\prime_{d_1,d_2,k_{+},k_{-}}:C^2(\pa D)\rightarrow L^2(\Sp^+_{\theta_c})$ such that for $h\in V_\ell$,
we have
\ben
F_{d_1,d_2,k_{+},k_{-}}[\pa D_h]- F_{d_1,d_2,k_{+},k_{-}}[\pa D]-F^\prime_{d_1,d_2,k_{+},k_{-}}[\pa D,h]=o(||h||_{C^2(\pa D)}).
\enn
The following theorem characterizes the Fr\'{e}chet derivative of the far-field operator $F_{d_1,d_2,k_{+},k_{-}}$.

\begin{theorem}\label{th:5.1}
Assume that $\pa D$ is a $C^2$-smooth boundary, the incident field is given by $u^i(x,d_1,d_2,k_{+})$
with $d_1,d_2\in\Sp^-_{\theta_c}$ and $v^0(x):=\sum^2_{i=1}u^0(x,d_i,k_{+},k_{-})$.
Let $u^s\in H^1_{loc}(\R^2\ba\ov{D})$
denote the scattered wave induced by $u^i(x,d_1,d_2,k_{+})$, which solves
the problem (\ref{eq:1}) with the boundary data $f=-v^0|_{\pa D}$. Then the operator
$F_{d_1,d_2,k_{+},k_{-}}$ is Fr\'{e}chet differentiable at $\pa D$ with the Fr\'{e}chet derivative given
by $F^\prime_{d_1,d_2,k_{+},k_{-}}[\pa D,h]=2\real(\ov{{u}^{\infty}}u^{\prime,\infty})$, where $h\in C^2(\pa D)$
and $u^{\prime,\infty}$ is the far-field pattern of $u^\prime\in H^1_{loc}(\R^2\ba\ov{D})$
solving the problem (\ref{eq:1}) with the boundary data ${\color{hw}{f=-[{\pa(v^0+u^s)}/{\pa\nu}](h\cdot\nu)}}$
on $\pa D$,
where $\nu$ denotes the outward unit normal vector on $\pa D$.
\end{theorem}

\begin{proof}
The statement of this theorem can be proved similarly as in \cite{K93,Hettlich_1995}.
\end{proof}

Next, we restrict ourselves to the case when the boundary $\pa D$ is a starlike curve, that is,
$\pa D$ has the form of the following suitable parametrization:
\be\label{eq:5.i}
\gamma=(a_1,a_2)+r(\theta)(\cos\theta,\sin\theta),\;\;\theta\in (0,2\pi],
\en
with its center at $(a_1,a_2)$.
In numerical computation, similarly to \cite{BLS}, we consider to use multiple sets of incident fields and thus rewrite (\ref{eq:5.1}) as the following perturbation equation:
\be\label{eq5.2}
F_{d^{(2)}_{1q},d^{(2)}_{2q},k_{+},k_{-}}[\gamma](\hat x)\approx|u_\delta^\infty(\hat x,d^{(2)}_{1q},d^{(2)}_{2q},k_+,k_-)|^2,\;q= 1,\ldots,n_d^{(2)},
\en
from a knowledge of the noisy phaseless far-field data
$u_\delta^\infty(\hat x,d^{(2)}_{1q},d^{(2)}_{2q},k_+,k_-)$, $d^{(2)}_{1q},d^{(2)}_{2q}\in\Sp^-_{\theta_c}$, $q=1,\ldots,n_d^{(2)}$, which satisfy that
\be\label{eq8}
\left\||u_\delta^\infty|^2-|u^\infty|^2\right\|_{L^2(\Sp^+_{\theta_c})}
\le\delta\left\||u^\infty|^2\right\|_{L^2(\Sp^+_{\theta_c})}
\en
with the noise level $\delta>0$.
Our Newton iteration method consists in using the Levenberg-Marquardt algorithm to solve
the linearized equation of (\ref{eq5.2})
\begin{align}\label{eq5.3}
F_{d^{(2)}_{1q},d^{(2)}_{2q},k_{+},k_{-}}[\gamma^{app}](\hat x)+F^\prime_{d^{(2)}_{1q},d^{(2)}_{2q},k_{+},k_{-}}[\gamma^{app},\Delta\gamma](\hat x)
\approx|u_\delta^\infty(\hat x,d^{(2)}_{1q},d^{(2)}_{2q},k_{+},k_{-})|^2,\;q=1,\ldots,n_d^{(2)},
\end{align}
for $\Delta\gamma$, where $\gamma^{app}=(a^{app}_1,a^{app}_2)+r^{app}(\theta)(\cos\theta,\sin\theta)$
is the approximation to $\gamma$.

Using the strategy in \cite{BLS}, $r^{app}$ is taken from a finite-dimensional subspace
$W_M\subset H^s(0,2\pi)$, $s\ge 0$, for practical numerical computation, where
\ben
W_M:=\{r\in H^s(0,2\pi):r(\theta)=\alpha_0+\sum^M_{l=1}\alpha_l\cos(l\theta)
+\alpha_{l+M}\sin(l\theta),\;\alpha_l\in\R,\;\;\text{for}\;l=0,\ldots,2M\}
\enn
with the norm
\ben
\|r\|^2_{H^s(0,2\pi)}:= 2\pi\alpha_0+\pi\sum^M_{l=1}(1+l^2)^s(\alpha_l^2+\alpha_{l+M}^2).
\enn
Then we seek the regularized solution $\Delta\gamma:=(\Delta a_1,\Delta a_2)+\Delta r(\theta)(\cos\theta,\sin\theta)$ of (\ref{eq5.3}) such that
$(\Delta a_1,\Delta a_2,\Delta r)$ is the solution of the minimization problem
\begin{align}\label{eq5.4}
\min_{\Delta a_1,\Delta a_2,\Delta r}
&\Bigg\{\sum^{n_d^{(2)}}_{q=1}\Big\|F_{d^{(2)}_{1q},d^{(2)}_{2q},k_{+},k_{-}}[\gamma^{app}](\hat x)
+F^\prime_{d^{(2)}_{1q},d^{(2)}_{2q},k_{+},k_{-}}[\gamma^{app},\Delta \gamma](\hat x)\notag\\
&-|u_{\delta}^{\infty}(\hat x,d^{(2)}_{1q},d^{(2)}_{2q},k_{+},k_{-})|^2\Big\|^2_{L^2(\Sp^+_{\theta_c})}
+\beta\left[\sum^2_{l=1}|\Delta a_l|^2+\|\Delta r\|^2_{H^s(0,2\pi)}\right]\Bigg\},
\end{align}
where the regularization parameter $\beta$ is chosen so that
\begin{align}\label{eq5.5}
&\sum^{n_d^{(2)}}_{q=1}\Big\|F_{d^{(2)}_{1q},d^{(2)}_{2q},k_{+},k_{-}}[\gamma^{app}](\hat x)+F^\prime_{d^{(2)}_{1q},d^{(2)}_{2q},k_{+},k_{-}}[\gamma^{app},\Delta\gamma](\hat x)
-|u_{\delta}^{\infty}(\hat x,d^{(2)}_{1q},d^{(2)}_{2q},k_{+},k_{-})|^2\Big\|^2_{L^2(\Sp^+_{\theta_c})}\notag \\
&\quad\quad\quad\quad\quad\quad=\rho^2 {\sum^{n_d^{(2)}}_{q=1}\Big\|F_{d^{(2)}_{1q},d^{(2)}_{2q},k_{+},k_{-}}[\gamma^{app}](\hat x)
-|u_{\delta}^{\infty}(\hat x, d^{(2)}_{1q},d^{(2)}_{2q},k_{+},k_{-})|^2\Big\|^2_{L^2(\Sp^+_{\theta_c})}}
\end{align}
for some parameter $\rho<1$
and $\beta$ is determined by using the bisection algorithm (see \cite{TH99}).
Thus the approximation $\gamma^{app}$ can be updated by $\gamma^{app}+\Delta\gamma$.

The stopping rule is provided by discrepancy principle (see \cite{TH99}), that is,
the iteration is stopped if $E_{k_{+},k_{-}}<\tau\delta$, where $\tau >1$ is a given constant and
the relative error $E_{k_{+},k_{-}}$ is defined by
\ben
E_{k_{+},k_{-}}=\frac{1}{n_d^{(2)}}\sum^{n_d^{(2)}}_{q=1}\frac{\Big\|F_{d^{(2)}_{1q},d^{(2)}_{2q},k_{+},k_{-}}[\gamma^{app}](\hat x)
-|u_{\delta}^{\infty}(\hat x, d^{(2)}_{1q},d^{(2)}_{2q},k_{+},k_{-})|^2\Big\|_{L^2(\Sp^+_{\theta_c})}}
{\Big\||u_{\delta}^{\infty}(\hat x, d^{(2)}_{1q},d^{(2)}_{2q},k_{+},k_{-})|^2\Big\|_{L^2(\Sp^+_{\theta_c})}}.
\enn

\begin{remark}\label{re3.2} {\rm For the numerical algorithm of this section,
we use the layered Green function method in \cite{Car17} to compute the synthetic data and the numerical solution in each iteration step.
To determine the location of the obstacle,
we use the measured data
$|u_{\delta}^\infty(\hat x_j,d^{(1)}_{1l},d^{(1)}_{2i},k^{(1)}_{+}, k^{(1)}_-)|$,
$j=1,\ldots,n_{f}$,
$l=1,\ldots,n^{(1)}_{d}$,
$i= 1,\ldots, n^{(1)}_{d}$, with fixed wave numbers $k^{(1)}_+>0$
and $k^{(1)}_-=\sqrt{n}k^{(1)}_{+}$.
For the iteration algorithm proposed in this section, we follow the idea in \cite{BLS}
to make use of the multi-frequency phaseless data $|u_{\delta}^{\infty}(\hat x_j,d^{(2)}_{1q},d^{(2)}_{2q},k^{(2)}_{+},k^{(2)}_{-})|$,
$j=1,\ldots,n_f$, $q=1,\ldots,n_d^{(2)}$,
%$k_{-m}=\sqrt{n}k_{+m}$,
with the wave numbers $k^{(2)}_+=k^{(2)}_{+,m}$ and $k^{(2)}_-={\sqrt{n}}k^{(2)}_{+,m}$, $m=1,\ldots,F$.
Here, $\hat{x}_j,d^{(1)}_{1l},d^{(1)}_{2i}$
are the same as in Section \ref{sec2},
$d^{(2)}_{1q},d^{(2)}_{2q}\in\Sp^{-}_{\theta_c}$
with $d^{(2)}_{1q}\ne d^{(2)}_{2q}$,
the multiple wave numbers satisfy $0<k^{(2)}_{+,1}<\cdots<k^{(2)}_{+,F}$ and
 $n$ is the refractive index.
Further, the norm $\|\cdot \|_{L^2(\Sp^+_{\theta_c})}$ can be approximated by
\[
\|f\|^2_{L^2(\Sp^+_{\theta_c})}\approx\frac{\pi-2\theta_c}{n_f}{\color{hw}{\sum^{n_f}_{j=1}}} |f(\hat x_j)|^2.
\]
}
\end{remark}
%\vspace{1em}

Based on the above discussions, our numerical algorithm for extend obstacles is presented in Algorithm \ref{A2}.
%for determining the obstacle buried in the lower half-space.

\begin{algorithm}
\caption{\textbf{Location and shape reconstruction of the extended obstacle}{\label{A2}}}
\SetAlgoLined
\KwIn{Noisy phaseless data $1$:\\
\qquad\quad$\;$ $|u_{\delta}^\infty(\hat x_j,d^{(1)}_{1l},d^{(1)}_{2i},k^{(1)}_{+}, \sqrt{n}k^{(1)}_{+})|,\;j=1,\ldots,n_{f}$,
$l=1,\ldots,n^{(1)}_{d}, i= 1,\ldots, n^{(1)}_{d}$.

\qquad\quad $\;$ Noisy phaseless data $2$:\\
\qquad\quad $\;$ $|u_{\delta}^{\infty}(\hat x_j,d^{(2)}_{1q},d^{(2)}_{2q},k^{(2)}_{+,m},{\sqrt{n}}k^{(2)}_{+,m})|$,
$j=1,\ldots,n_{f},\;q=1,\ldots,n^{(2)}_{d}$, $m=1,\ldots,F$.}
\KwOut{The location and shape of the obstacle.}

Set $k_+=k^{(1)}_+$ and $k_-=\sqrt{n}k^{(1)}_+$.

Locating the obstacle by Algorithm \ref{A1} with noisy phaseless data $1$.

Given the parameters $\tau,\rho$, choose the initial guess $\gamma^{app}$ to be a circle with radius $r_0$,
whose center is the local maximum of the imaging result by Algorithm \ref{A1}.\\
\For{$k^{(2)}_{+}=k^{(2)}_{+,1},\ldots,k^{(2)}_{+,F}$}
{Set $k_+=k^{(2)}_{+}$ and $k_-=\sqrt{n}k^{(2)}_{+}$.

\While {$E_{k_{+},k_-}\geq\tau\delta$}
{Use the strategy (\ref{eq5.5}) to solve (\ref{eq5.4}) with noisy phaseless data $2$ to
update the approximation $\gamma^{app}$ as $\gamma^{app}=\gamma^{app}+\Delta\gamma$.
}
}
\end{algorithm}

\begin{remark}\label{re3.3} {\rm
Algorithm \ref{A2} can be extended to reconstruct extended scatterers which consist of several
connected components.
In this case, we assume that each component has the parametrization given in (\ref{eq:5.i}).
}
\end{remark}

\section{Numerical experiments}\label{sec4}

\subsection{Locating multiple small scatterers}

We first present several numerical examples to illustrate the applicability of the direct imaging
algorithm (i.e., Algorithm \ref{A1}) for imaging small scatterers. To generate the synthetic data,
the direct scattering problem is solved by the layered Green function method proposed in \cite{Car17}.
As for the far-field data, {\color{hw}it is} measured with $256$ incident and observed directions
which are uniformly distributed on $\Sp^{-}_{\theta_c}$ and $\Sp^{+}_{\theta_c}$, respectively, {\color{hw} that is, $n_f=256$ and $n^{(1)}_d= 256$}.
Further, the noisy phaseless data $u^{\infty}_{\delta}(\hat x_j)$, $j=1,\ldots,n_f$, {\color{hw}are} given as
\begin{align}\label{eq6.1}
|u^{\infty}_{\delta}(\hat x_j)|^2=|u^{\infty}(\hat x_j)|^2
+\delta\frac{\xi_j}{\sqrt{\sum^{n_f}_{p=1}|\xi_p|^2}}
\sqrt{\sum\nolimits^{n_f}_{p=1}|u^{\infty}(\hat x_p)|^4},\;
j=1,\ldots,n_f,
\end{align}
where $\delta$ is noise level and $\xi={(\xi_j)}_{j=1,\ldots, n_f}$ with $\xi_j$ being standard normal distribution.

%\vspace{1em}
\textbf{Example 1: Locating a small obstacle.}
%\vspace{0.5em}
We consider the scattering problem by a circle buried in the lower half-space. Our aim is to show
that the numerical performance of the imaging function $I_A(z,k_+,k_-)$ is consistent with our analysis
in Section \ref{sec2}, as shown in Figure \ref{fig1}. In Figure \ref{fig1}, we consider the circle with
radius $0.1$ and center at $(-3,-3)$, and the sampling region is chosen to be $[-4.5,4.5]\times[-4.5,4.5]$.
Figure \ref{fig1}(a) presents the exact position of the small circle.
Figure \ref{fig1}(b) shows the reconstruction result from the measured data
with $10\%$ noise in the case $k_{+}<k_{-}$ with $k_{+}=10\pi$ and $k_-=1.45k_{+}$.
Figure \ref{fig1}(c) presents the corresponding imaging result from the measured data with $10\%$ noise
in the case $k_{+}>k_{-}$ with $k_{+}=15\pi$ and $k_{-}=k_{+}/1.5$.
It is observed that $I_A(z,k_+,k_-)$ takes a large value in the neighborhood of $(-3,-3)$ and $(3,3)$, which is
consistent with the analysis in Section \ref{sec2}. With the aid of the a priori information
that the obstacle is buried in the lower half-space, we can determine that the position of the small obstacle is $(-3,3)$.

%\vspace{-0.05em}
\begin{figure}[htbp]
\centering
\subfigure[]
{
\begin{minipage}[t]{0.26\textwidth}
\centering
\includegraphics[width=\textwidth]{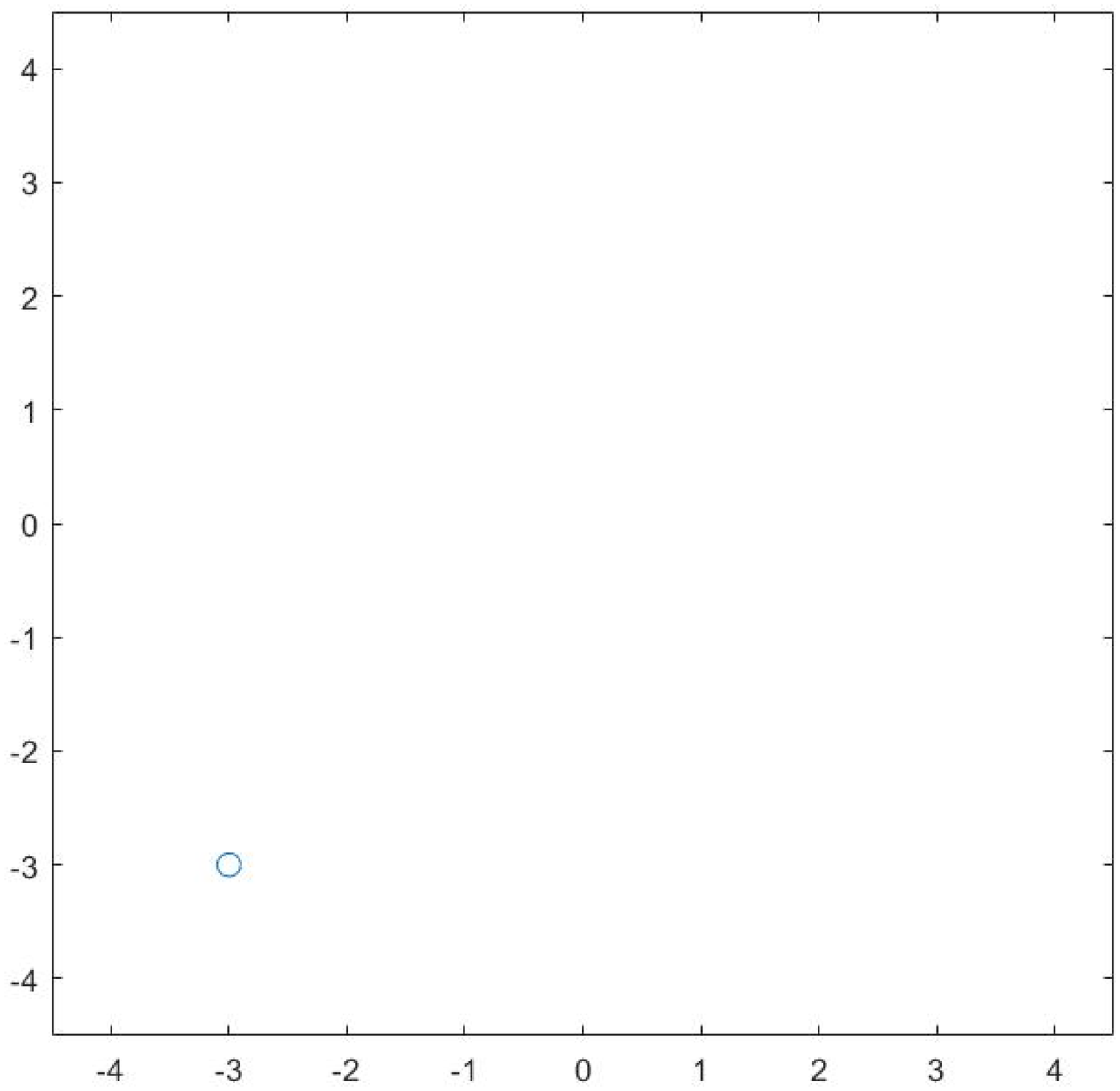}
%\caption{fig1}
\end{minipage}%
}%
\hfill
\subfigure[]
{
\begin{minipage}[t]{0.32\textwidth}
\centering
\includegraphics[width=\textwidth]{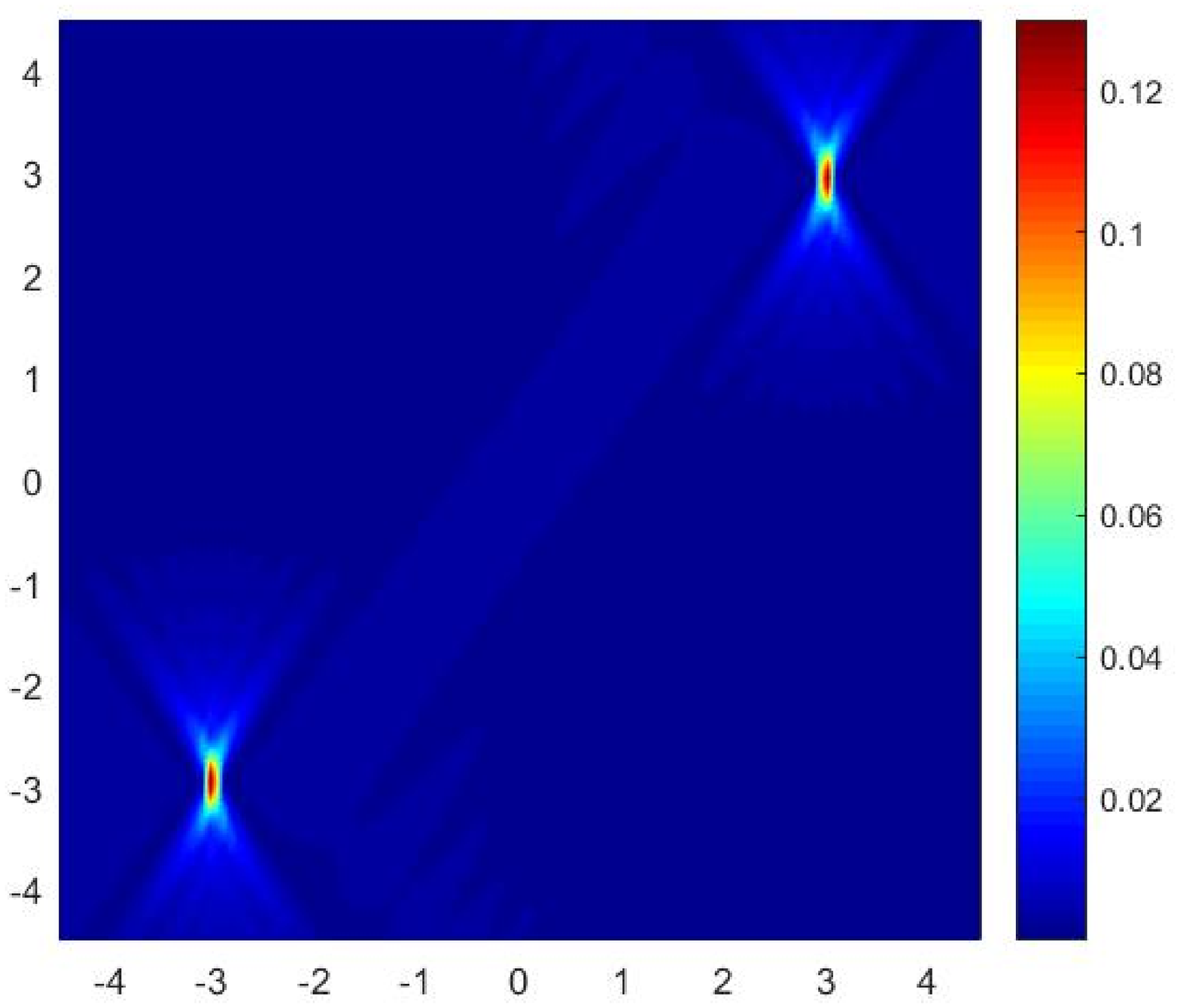}
%\caption{fig1}
\end{minipage}%
}%
\hfill
\subfigure[]
{
\begin{minipage}[t]{0.32\textwidth}
\centering
\includegraphics[width=\textwidth]{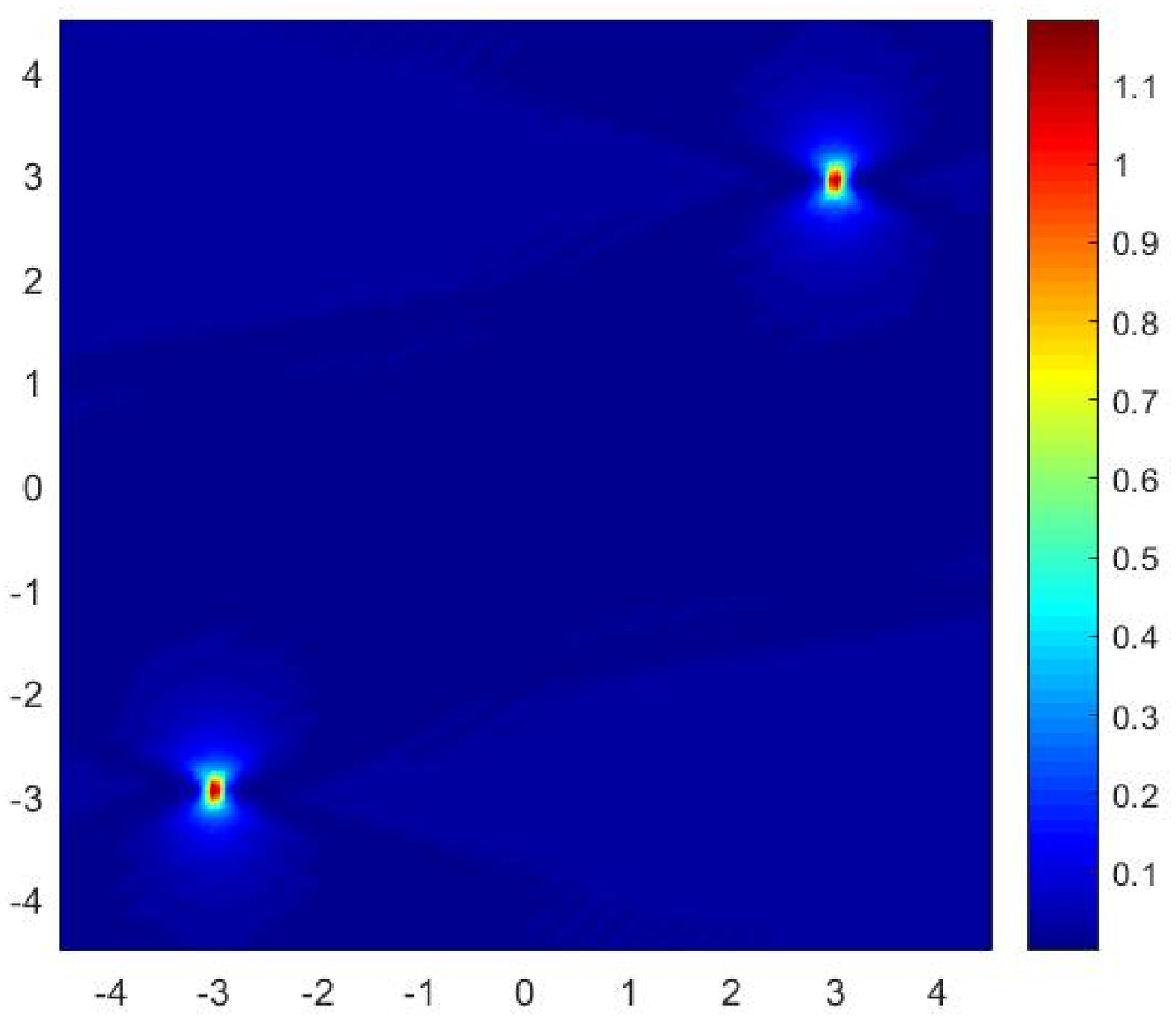}
%\caption{fig2}
\end{minipage}%
}%
\centering
\caption{Imaging results of a small scatterer by Algorithm \ref{A1} using phaseless far-field data
with 10\% noise: (a) True small scatterer, (b)-(c) Imaging results of a small scatterer at $k_+=10\pi$
and $k_-=1.45k_{+}$, and at $k_+=15\pi$ and $k_{-}=k_{+}/1.5$, respectively.
}\label{fig1}
\end{figure}

%\vspace{-0.05em}

%\vspace{0.5em}
\textbf{Example 2: Locating multiple small anomalies in the case $k_{+}<k_{-}$.}
%\vspace{0.5em}
Consider three small circles with radius $0.1$ and centers at $(-2,-7),(0,-6),(3,-5)$, respectively.
Here, we choose $k_{+}=10\pi$ and $k_{-}=1.45k_{+}$. The sampling region is taken to be $[-4.5,4.5]\times[-9,0]$. Figure \ref{fig2}(a) gives the actual position of the three small circles. Figures \ref{fig2}(b) and
\ref{fig2}(c) show the imaging results of the imaging function (\ref{eq:2.5}) with using the measured
data with $5\%$ noise and with $10\%$ noise, respectively.
It is clearly seen from Figure \ref{fig2} that the location and number of the three small scatterers {\color{hw}{are}} very well retrieved.

%\vspace{-0.05em}
\begin{figure}[htbp]
\centering
\subfigure[]
{
\begin{minipage}[t]{0.26\textwidth}
\centering
\includegraphics[width=\textwidth]{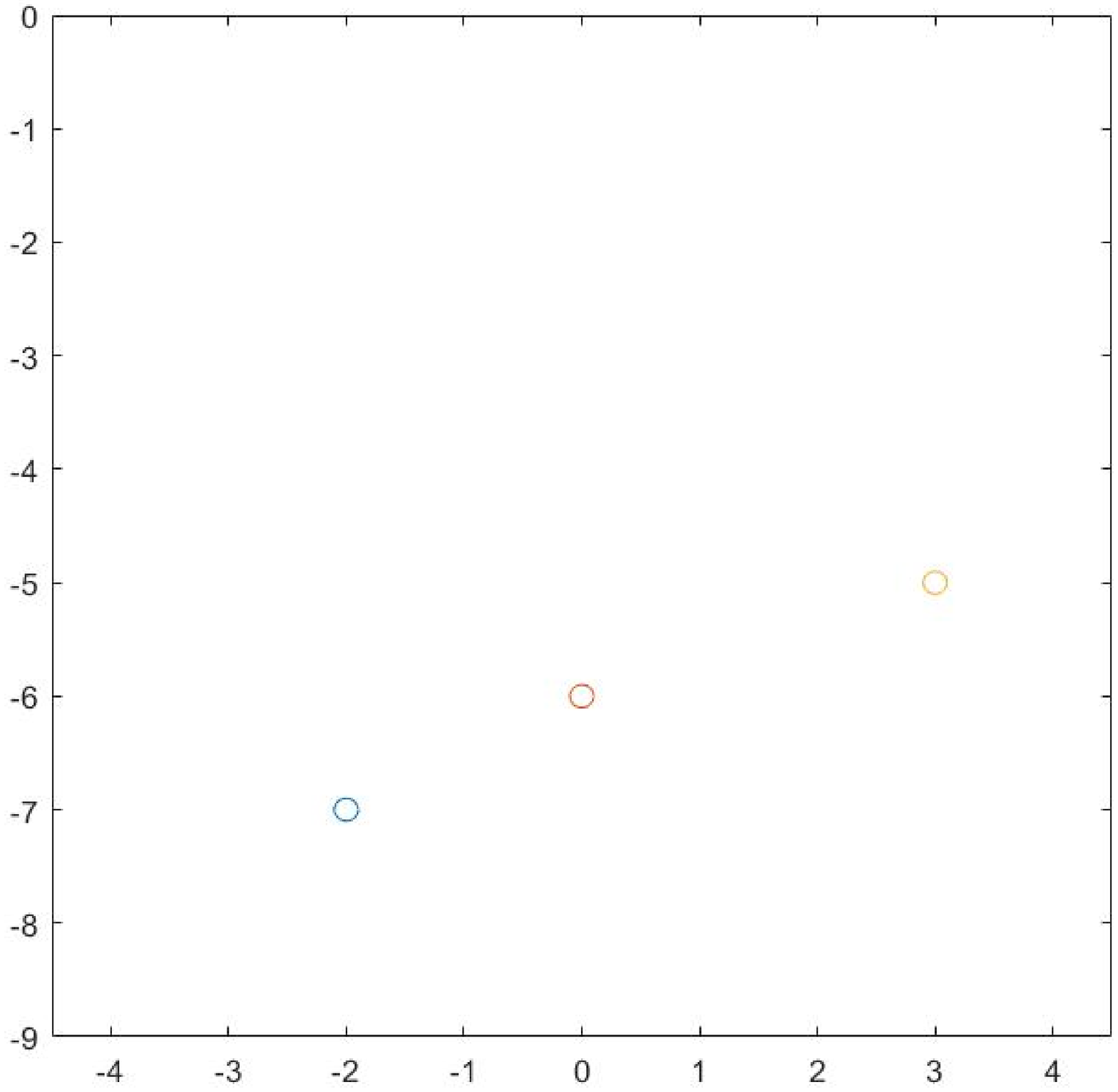}
%\caption{fig1}
\end{minipage}%
}%
\hfill
\subfigure[]
{
\begin{minipage}[t]{0.32\textwidth}
\centering
\includegraphics[width=\textwidth]{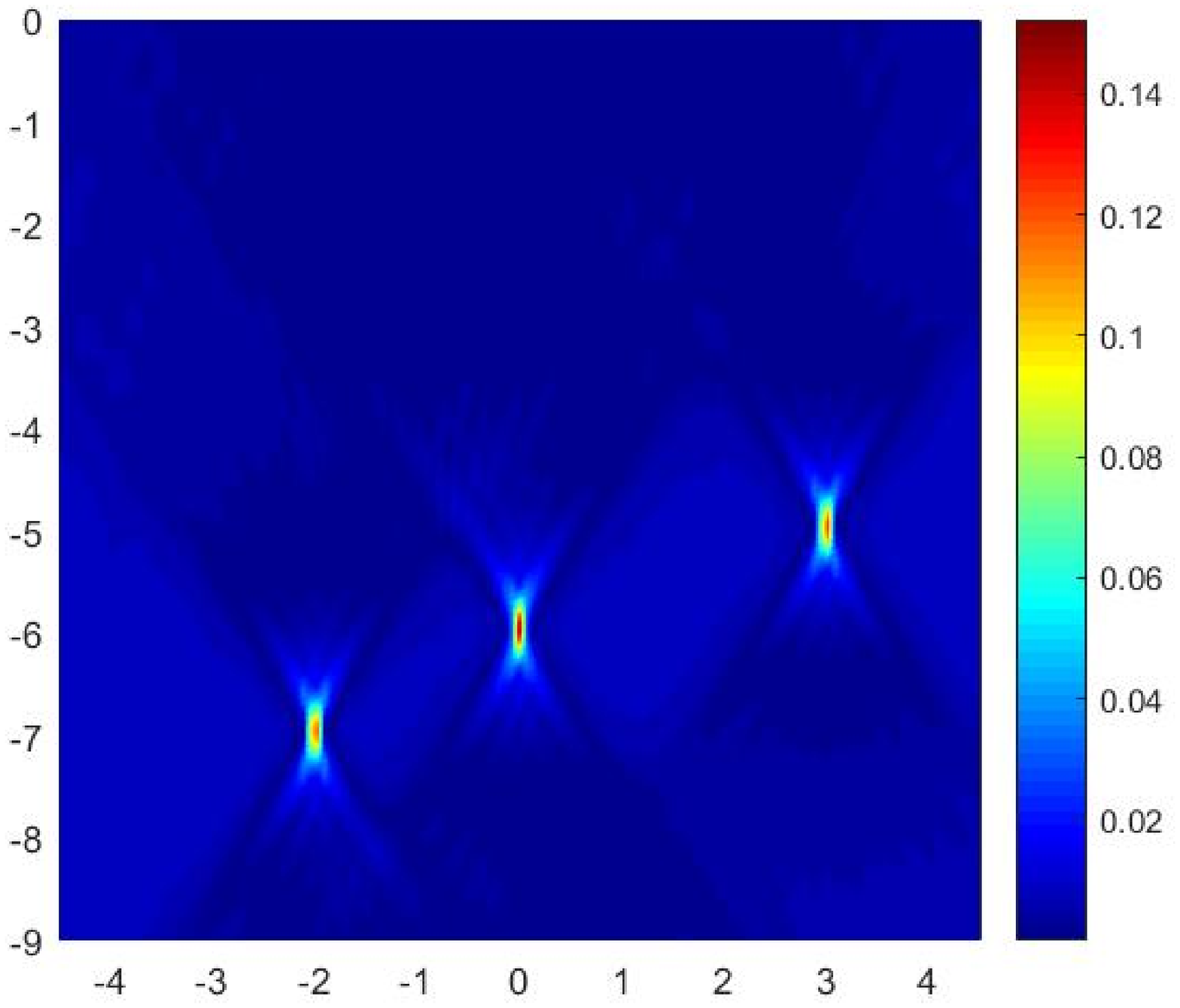}
%\caption{fig1}
\end{minipage}%
}%
\hfill
\subfigure[]
{
\begin{minipage}[t]{0.32\textwidth}
\centering
\includegraphics[width=\textwidth]{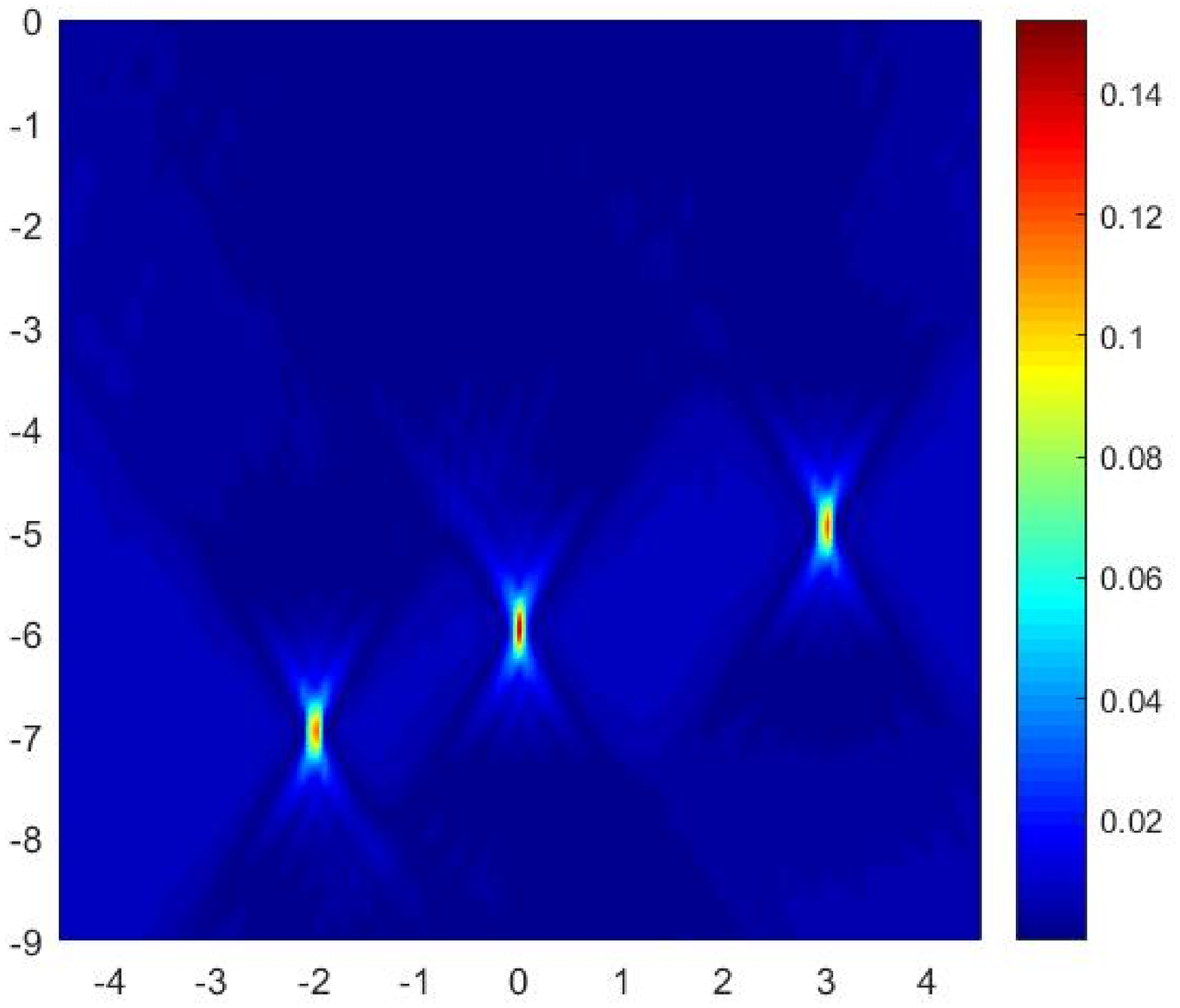}
%\caption{fig2}
\end{minipage}%
}%
\centering
\caption{Imaging results of multiple small scatterers by Algorithm \ref{A1} with phaseless far-field
data with (b) $5\%$ noise and (c) $10\%$ noise for the case $k_{+}=10\pi$ and $k_{-}=1.45k_{+}$,
where (a) shows the true scatterers.
}\label{fig2}
\end{figure}

%\vspace{0.5em}
\textbf{Example 3: Locating multiple small anomalies in the case $k_{+}>k_{-}$.}
%\vspace{0.5em}
Consider three small circles with radius $0.1$ and centers at $(-3,-8),(0,-2),(3,-5)$, respectively.
The wave numbers are chosen as $k_{+}=15\pi$ and $k_{-}=k_{+}/1.5$.
The sampling region is again taken to be $[-4.5,4.5]\times[-9,0]$. Figure \ref{fig3} presents
the exact position of the three multiple small anomalies and the imaging results given by the imaging
function $I_A(z,k_+,k_-)$ from the measured data with $5\%$ noise and with $10\%$ noise, respectively.
Similar to the case $k_{+}<k_{-}$ in Example 2, the location and number of the three unknown small scatterers
are satisfactorily obtained.

\begin{figure}[htbp]
\centering
\subfigure[]
{
\begin{minipage}[t]{0.26\textwidth}
\centering
\includegraphics[width=\textwidth]{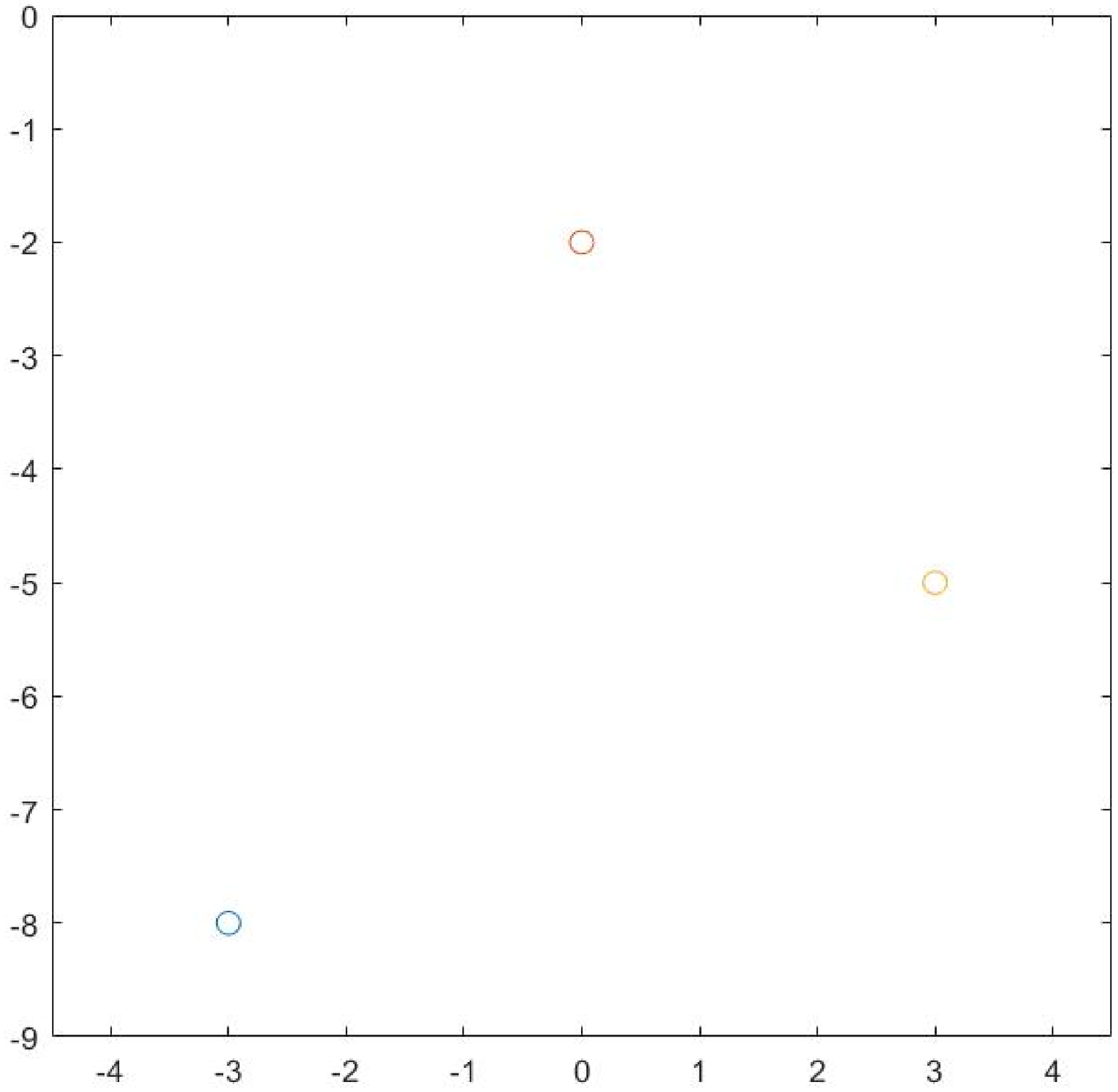}
%\caption{fig1}
\end{minipage}%
}%
\hfill
\subfigure[]
{
\begin{minipage}[t]{0.32\textwidth}
\centering
\includegraphics[width=\textwidth]{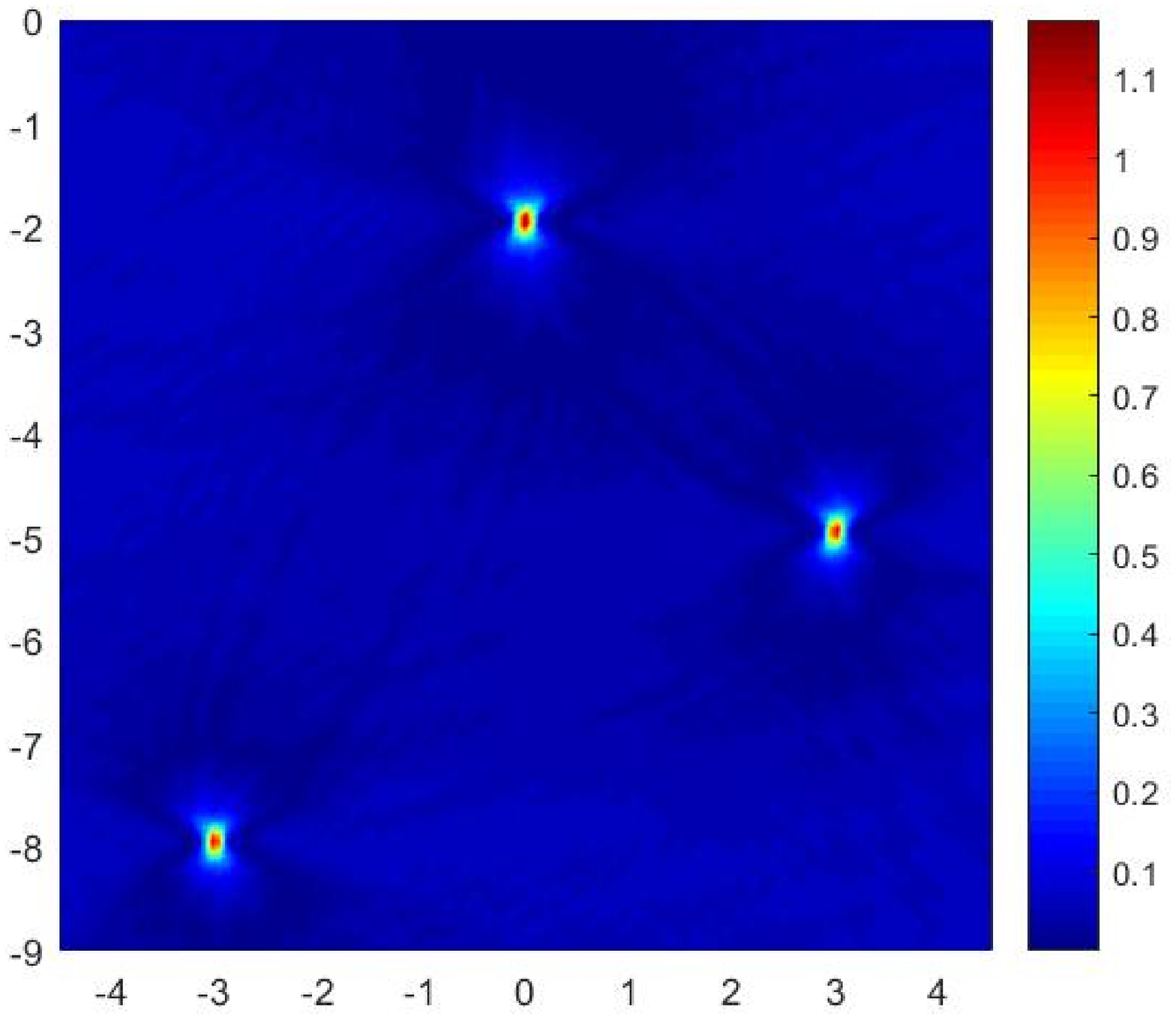}
%\caption{fig1}
\end{minipage}%
}%
\hfill
\subfigure[]
{
\begin{minipage}[t]{0.32\textwidth}
\centering
\includegraphics[width=\textwidth]{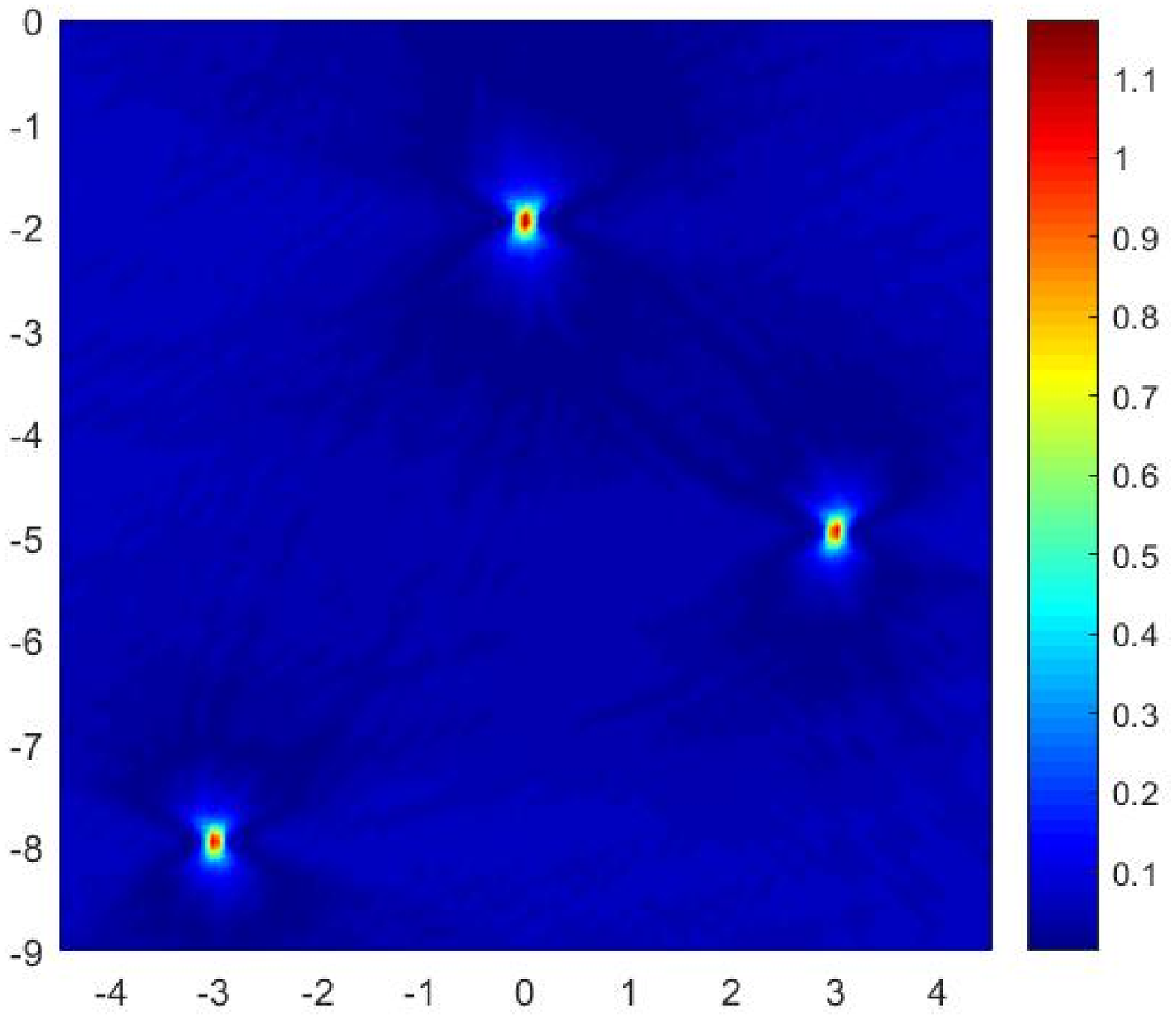}
%\caption{fig2}
\end{minipage}%
}%
\centering
\caption{Imaging results of multiple small scatterers by Algorithm \ref{A1} with phaseless far-field
data with (b) $5\%$ noise and (c) $10\%$ noise for the case $k_{+}=15\pi$ and $k_{-}=k_{+}/1.5$,
where (a) shows the true scatterers.}\label{fig3}
\end{figure}

\subsection{Location and shape reconstruction of extended obstacles}

We now carry out numerical implementation for Algorithm \ref{A2} presented in Section \ref{sec3}.
Shape reconstruction of the obstacles buried in the lower half-space will be considered in two cases:
$k_{+}>k_{-}$ and $k_{+}<k_{-}$.  The corresponding far-field pattern is computed by the layered Green
function method given in \cite{Car17} with the number of collocation points doubled in order to
avoid inverse crime. {\color{hw}Noisy phaseless data with noise level $\delta = 4\%$ are} simulated by using (\ref{eq6.1}), which
satisfy the condition (\ref{eq8})  approximately.
In all numerical examples,
we choose the parameters {\color{hw}$n_f=256$, $n^{(1)}_d=256$, $s=1.6$, $\rho=0.935$, $M=25$, $r_0=0.35$ and $\tau= 1.45$ in Algorithm \ref{A2}.} 
In the case $k_{+}>k_{-}$,  we choose
the refractive index $n=1/4$ and use multi-frequency data  with $k^{(2)}_{+}=1.5,3,6,10,14,18,22,26,30$. {\color{hw}{And the noisy phaseless far-field data 2 are generated by the incident waves
with three different sets of incident directions}} (i.e., $n_d^{(2)}=3$) with $d^{(2)}_{11}=(\cos\theta_c,-\sin\theta_c)$, $d^{(2)}_{21}=(-\cos\theta_c,-\sin\theta_c)$,
$d^{(2)}_{12}=(\cos\theta_c,-\sin\theta_c)$, $d^{(2)}_{22}=(\cos(-{3\pi}/{4}+{\theta_c}/2),\sin(-{3\pi}/4+{\theta_c}/2))$,
$d^{(2)}_{13}=(0,-1)$, $d^{(2)}_{23}=(\cos(-{\pi}/{4}-{\theta_c}/2),\sin(-\pi/4-{\theta_c}/2))$.
In the case $k_{+}<k_{-}$, the refractive index is chosen as $n={{1.45}^2}$, and we use multi-frequency
{\color{hw}phaseless data 2} with $k^{(2)}_{+}=0.8,1.5,2,3,4,5,7,11,13$. {\color{hw}{And the noisy phaseless far-field data 2 are  measured
by using four different sets of incident directions}} (i.e. $n^{(2)}_d=4$) with
$d^{(2)}_{11}=(\cos({\pi}/{300}),-\sin({\pi}/{300}))$, $d^{(2)}_{21}=(-{\sqrt{2}}/{2}, -{\sqrt{2}}/{2})$,
$d^{(2)}_{12}=(0,-1)$, $d^{(2)}_{22}=({\sqrt{2}}/{2}, -{\sqrt{2}}/{2})$, $d^{(2)}_{13}=(\cos({\pi}/{400}),-\sin({\pi}/{400}))$,
$d^{(2)}_{23}=({\sqrt{2}}/2, -{\sqrt{2}}/2)$, $d^{(2)}_{14}=(0,-1)$, $d^{(2)}_{24}=(-{\sqrt{2}}/2, -{\sqrt{2}}/2)$.
We recall that, according to the settings in Section \ref{sec3},
if the wave number $k_+$ in $\R^2_+$ is given by
$k_+=k^{(i)}_+$ ($i=1,2$)
then the wave number $k_-$ in $\R^2_-$ is given by
$k_-=k^{(i)}_-:=\sqrt{n}k^{(i)}_+$.
The parametrization of the test curves for the boundary $\pa D$ is given in Table \ref{table1}.

\begin{table}[htbp]
\centering
%\vspace{1em}
\begin{tabular}{ll}
\hline Type & Parametrization \\
\hline
Ellipse   & $(\cos t-5,1.35\sin t-6), t\in[0,2\pi]$ \\
Apple shaped     &   $[0.5+0.4\cos t+0.1\sin (2t)/(1+0.7\cos t)](\cos t, \sin t)-(0,4),t\in[0,2\pi]$  \\
Rounded triangle   & $(1+0.15\cos(3t))(\cos t, \sin t)-(2,2), t\in[0,2\pi] $\\
Rounded square     &$(0.6\cos^3(t)+ 0.6\cos t+1.5, 0.6\sin^3(t)+ 0.6\sin t-4.2),t\in[0,2\pi]$\\
\hline
\end{tabular}
%\vspace{1em}
\caption{Parametrization of the curves}\label{table1}
\end{table}

%\vspace{0.5em}
\textbf{Example 4: Reconstruction of an obstacle in the case $k_+>k_-.$}
%
%\vspace{0.25em}
%
Consider the inverse problem for reconstructing the apple-shaped obstacle in the case $k_{+}>k_{-}$.
Figure \ref{fig4}(a) presents the imaging result in the sampling region $[-2.5,2.5]\times[-6.5,-1.5]$
by the direct imaging algorithm with $k^{(1)}_+=10$, whose local maximum is at $(0.16,-3.69)$.
Figures \ref{fig4}(b) and \ref{fig4}(c) present the initial curve and the reconstruction result at $k^{(2)}_+=30$,
respectively,
where the solid line represents the exact curve. It can be seen
from Figure \ref{fig4}(c) that the location and shape of
the obstacle are satisfactorily reconstructed.

\begin{figure}[htbp]
\centering
\subfigure[]
{
\begin{minipage}[t]{0.36\textwidth}
\centering
\includegraphics[width=\textwidth]{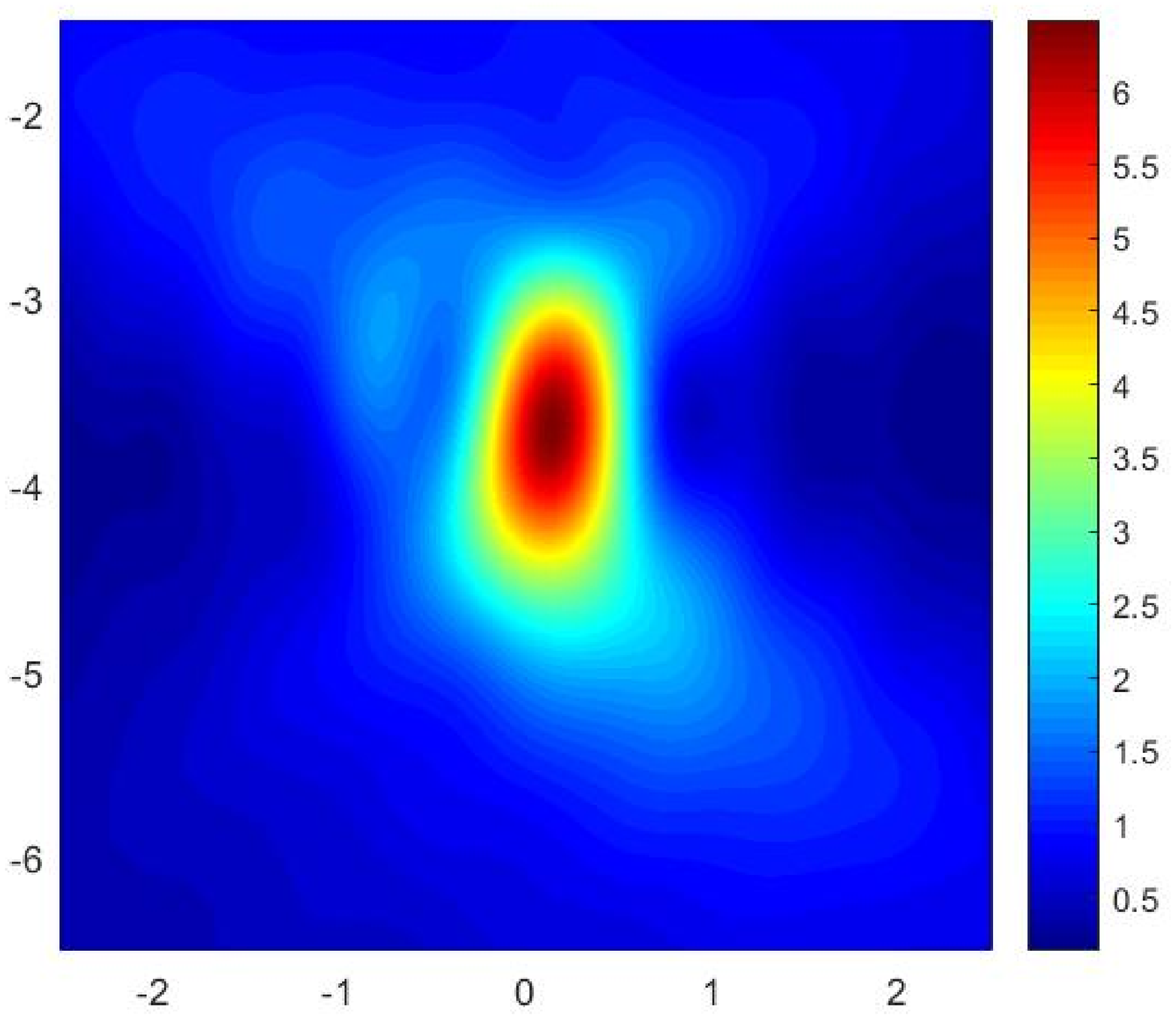}
%\caption{fig1}
\end{minipage}%
}%
\hfill
\subfigure[]
{
\begin{minipage}[t]{0.28\textwidth}
\centering
\includegraphics[width=\textwidth]{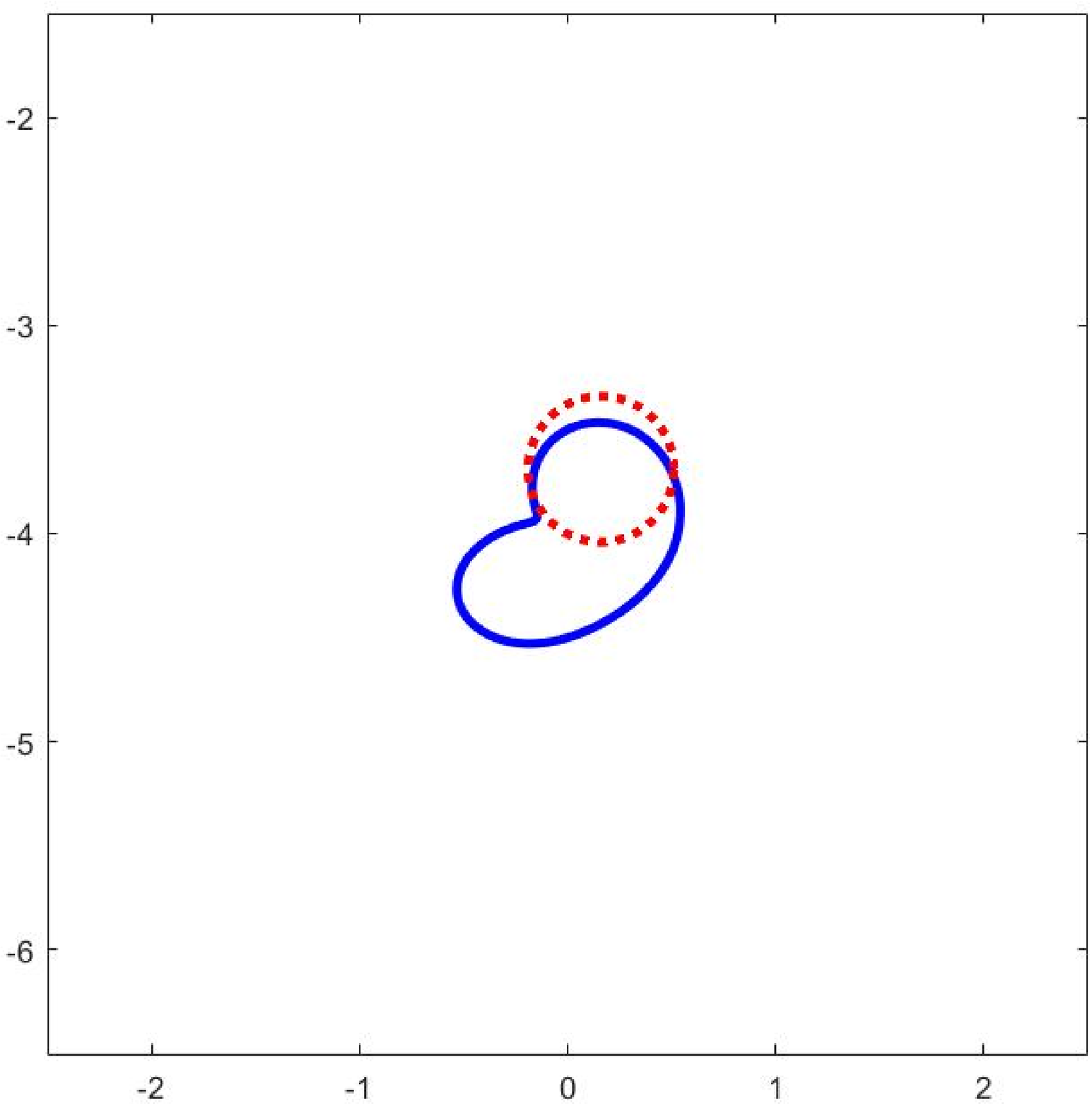}
%\caption{fig1}
\end{minipage}%
}
\hfill
\subfigure[]
{
\begin{minipage}[t]{0.28\textwidth}
\centering
\includegraphics[width=\textwidth]{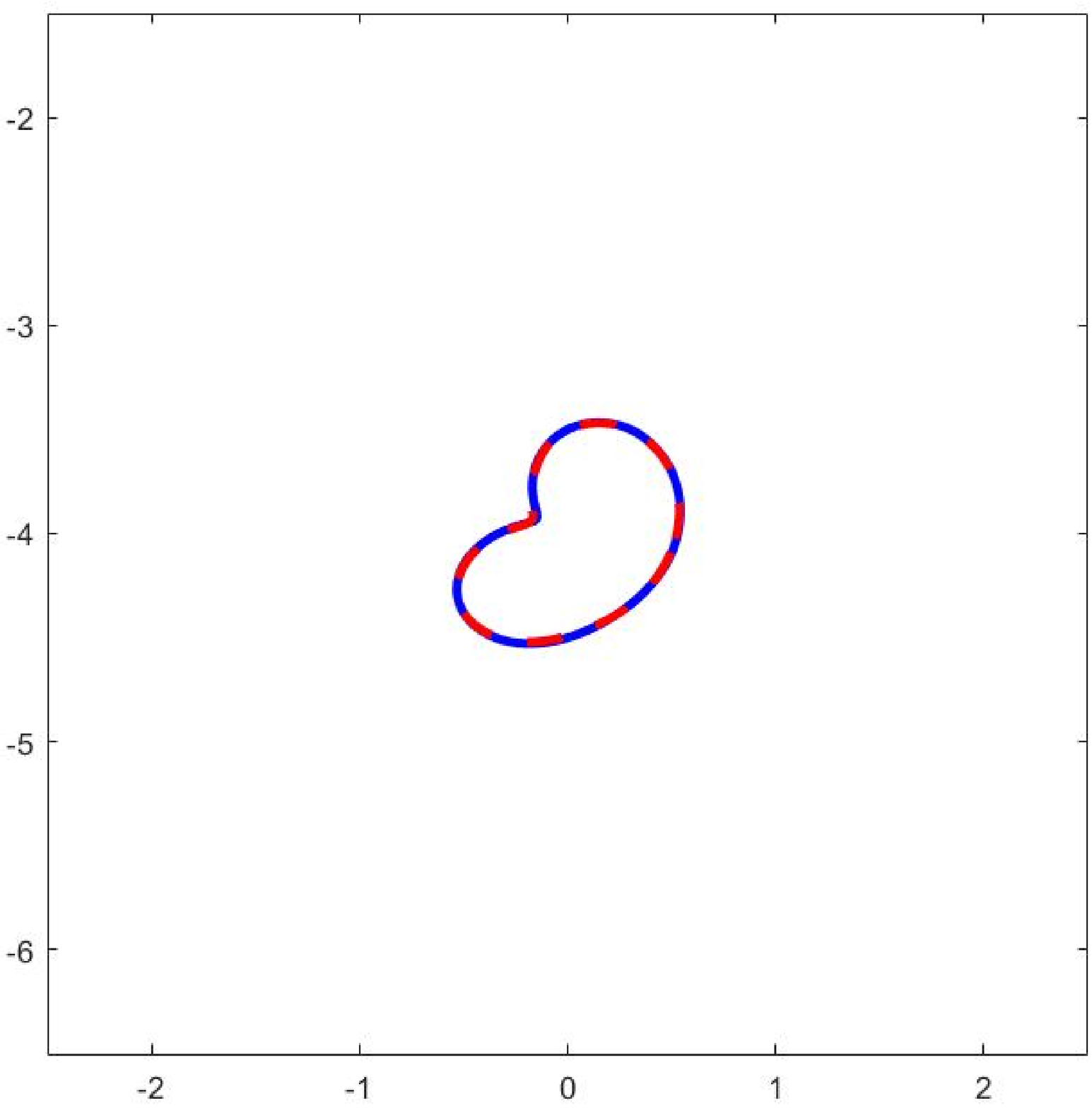}
%\caption{fig2}
\end{minipage}%
}%
\centering
\caption{Location and shape reconstruction of an apple-shaped obstacle from the phaseless far-field data
with $4\%$ noise in the case $k_+>k_-$: (a) The reconstruction result by Algorithm \ref{A1} at $k^{(1)}_+=10$ and
$k^{(1)}_-=k^{(1)}_+/2$, (b) The initial curve for Algorithm \ref{A2}, (c) The reconstructed obstacle
by Algorithm \ref{A2} at $k^{(2)}_+=30$ and $k^{(2)}_-=k^{(2)}_+/2$.}\label{fig4}
\end{figure}

%\vspace{0.5em}
\textbf{Example 5: Reconstruction of an obstacle in the case $k_{+}<k_{-}$.}
%
%\vspace{0.3em}
Consider the inverse problem in the case $k_{+}<k_{-}$,
where the obstacle is
the same as in Example 4.
Figure \ref{fig5}(a) presents the imaging result by the direct imaging method with $k^{(1)}_+=10$,
whose local maximum is at $(0.16,-3.69)$. Figures \ref{fig5}(b) and \ref{fig5}(c) present the initial curve and
the reconstruction result at $k^{(2)}_+=13$, respectively, where the solid line represents the exact curve.
It can be seen from Figure \ref{fig5}(c) that only the upper part of the obstacle can be satisfactorily reconstructed.

\begin{figure}[htbp]
\centering
\subfigure[]
{
\begin{minipage}[t]{0.36\textwidth}
\centering
\includegraphics[width=\textwidth]{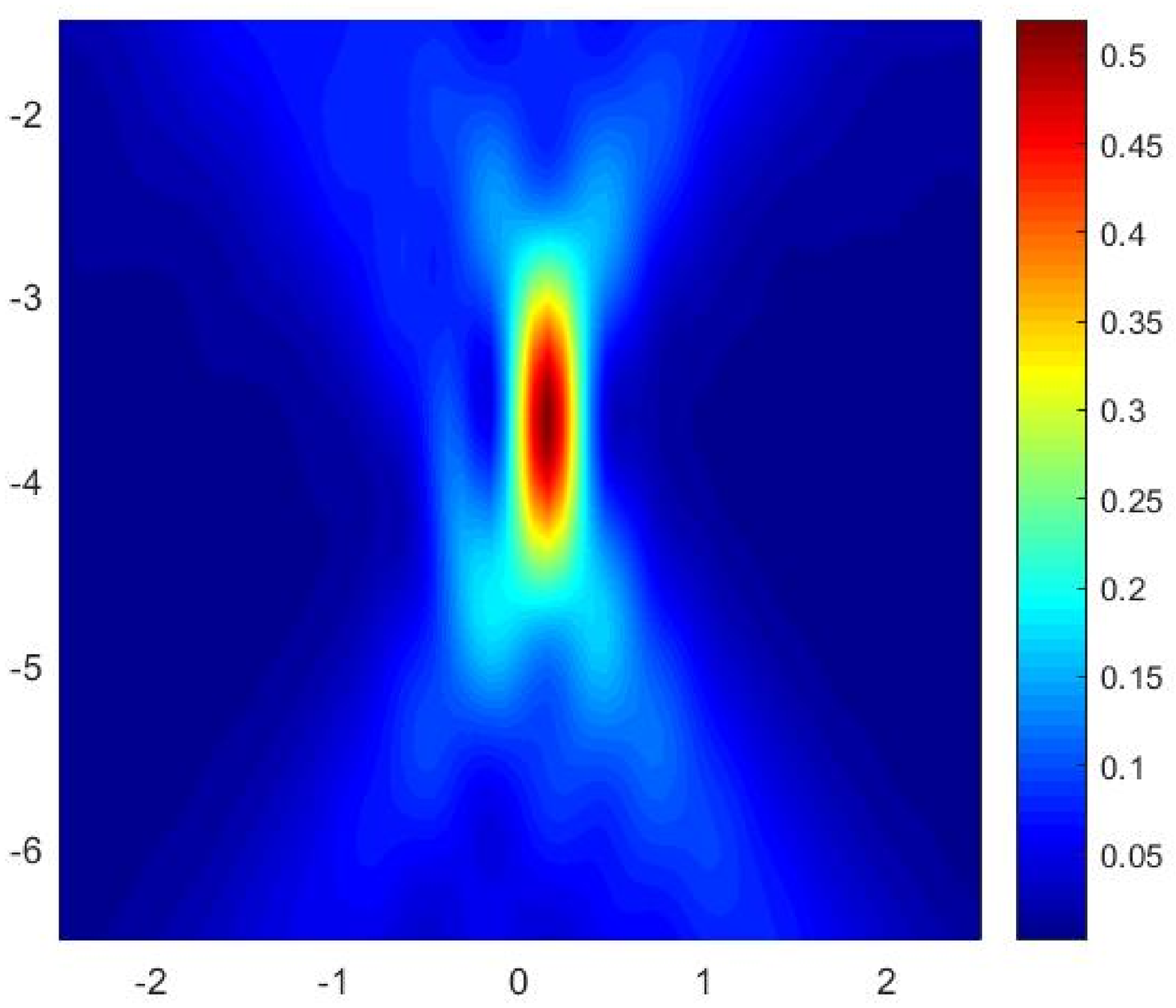}
%\caption{fig1}
\end{minipage}%
}%
\hfill
\subfigure[]
{
\begin{minipage}[t]{0.28\textwidth}
\centering
\includegraphics[width=\textwidth]{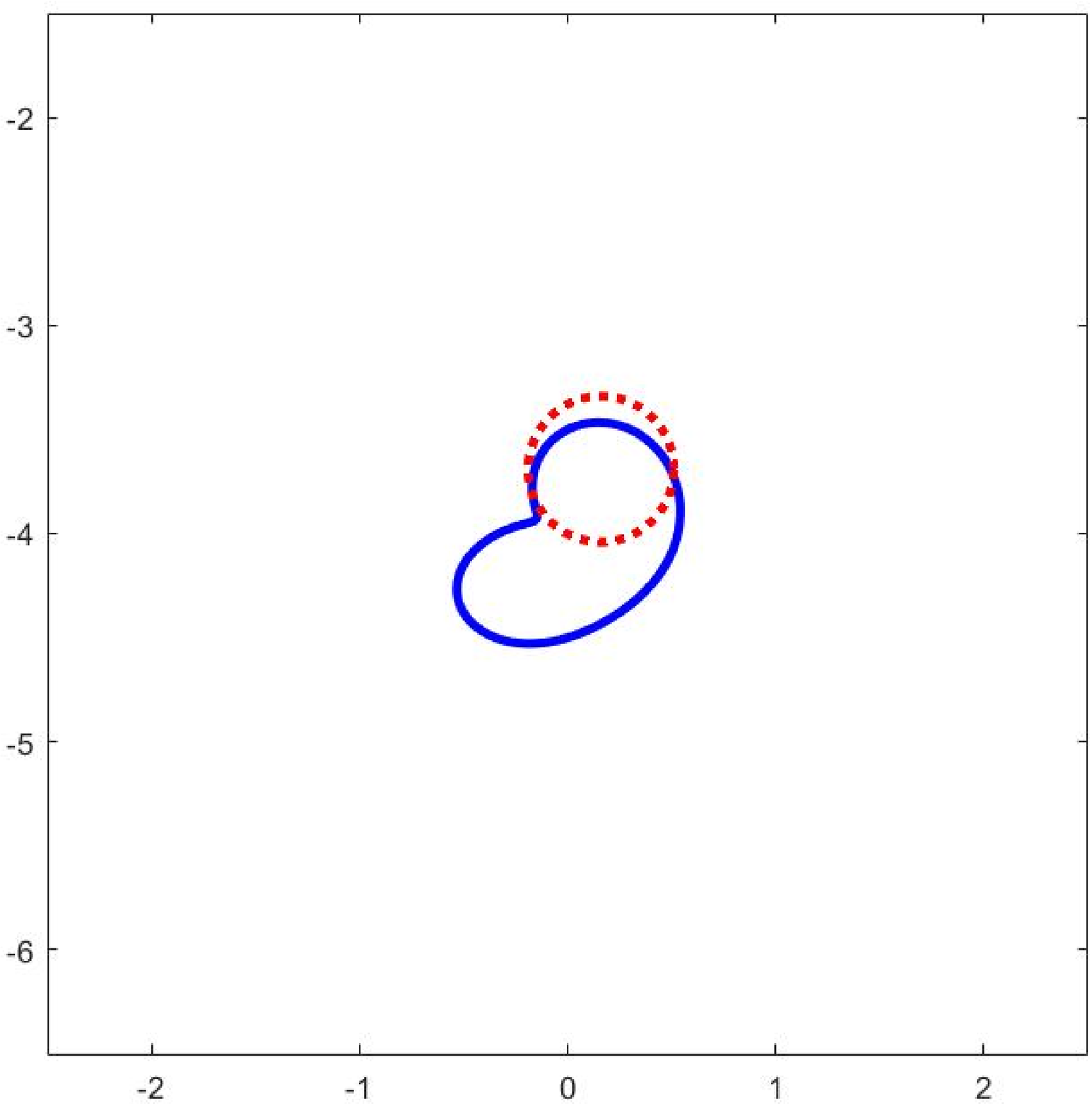}
%\caption{fig1}
\end{minipage}%
}%
\hfill
\subfigure[]
{
\begin{minipage}[t]{0.28\textwidth}
\centering
\includegraphics[width=\textwidth]{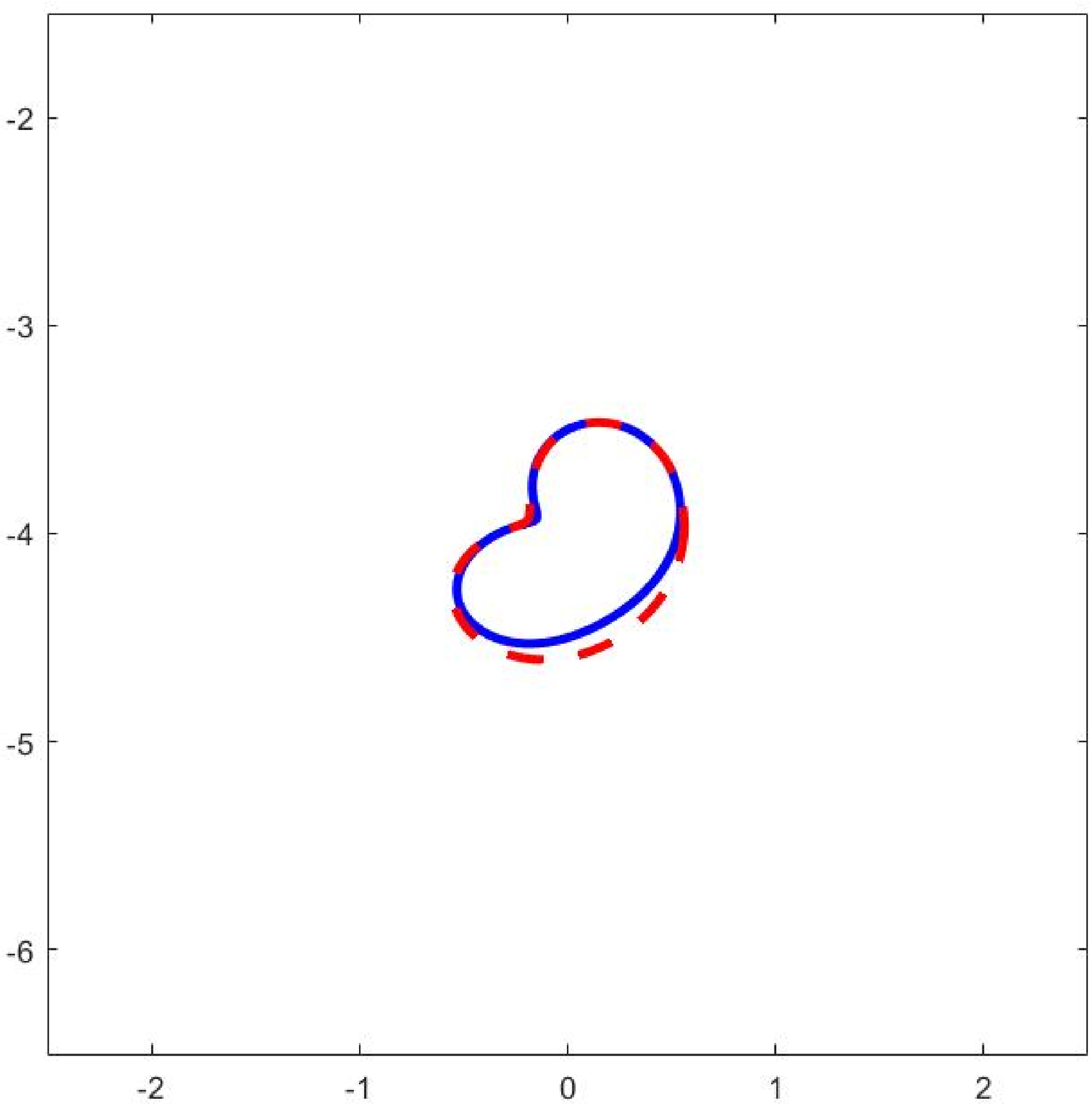}
%\caption{fig2}
\end{minipage}%
}%
\centering
\caption{Location and shape reconstruction of an apple-shaped obstacle from the phaseless far-field data
with $4\%$ noise in the case $k_+<k_-$: (a) The reconstruction result by Algorithm \ref{A1} at $k^{(1)}_+=10$ and
$k^{(1)}_{-}=1.45k^{(1)}_{+}$, (b) The initial curve for Algorithm \ref{A2}, (c) The reconstructed obstacle
by Algorithm \ref{A2} at $k^{(2)}_+=13$ and $k^{(2)}_{-}=1.45k^{(2)}_{+}$.}\label{fig5}
\end{figure}

%\vspace{0.5em}
\textbf{Example 6: Reconstruction of multiple obstacles in the case $k_{+}>k_{-}$.}
%
%\vspace{0.3em}
We now consider the inverse problem for reconstructing multiple obstacles consisting of an ellipse-shape,
a rounded triangle-shape and a rounded square-shape in the case $k_{+}>k_{-}$.
Figure \ref{fig6}(a) presents the imaging result by Algorithm \ref{A1} with $k^{(1)}_{+}=30$, whose local maximums
are at $(-1.97,-1.25),(-4.94,-4.85),(1.47,-3.13)$, respectively. Figures \ref{fig6}(b) and \ref{fig6}(c) present the initial curve
and the reconstruction result at $k^{(2)}_{+}=30$,
respectively, where the solid line represents the exact curve.
It is seen from Figure \ref{fig6}(c) that the location and shape of the rounded triangle-shaped and ellipse-shaped obstacles
are satisfactorily reconstructed. However, the lower part of the rounded square-shaped
obstacle is not very accurately reconstructed compared with its upper part.

\begin{figure}[htbp]
\centering
\subfigure[]
{
\begin{minipage}[t]{0.36\textwidth}
\centering
\includegraphics[width=\textwidth]{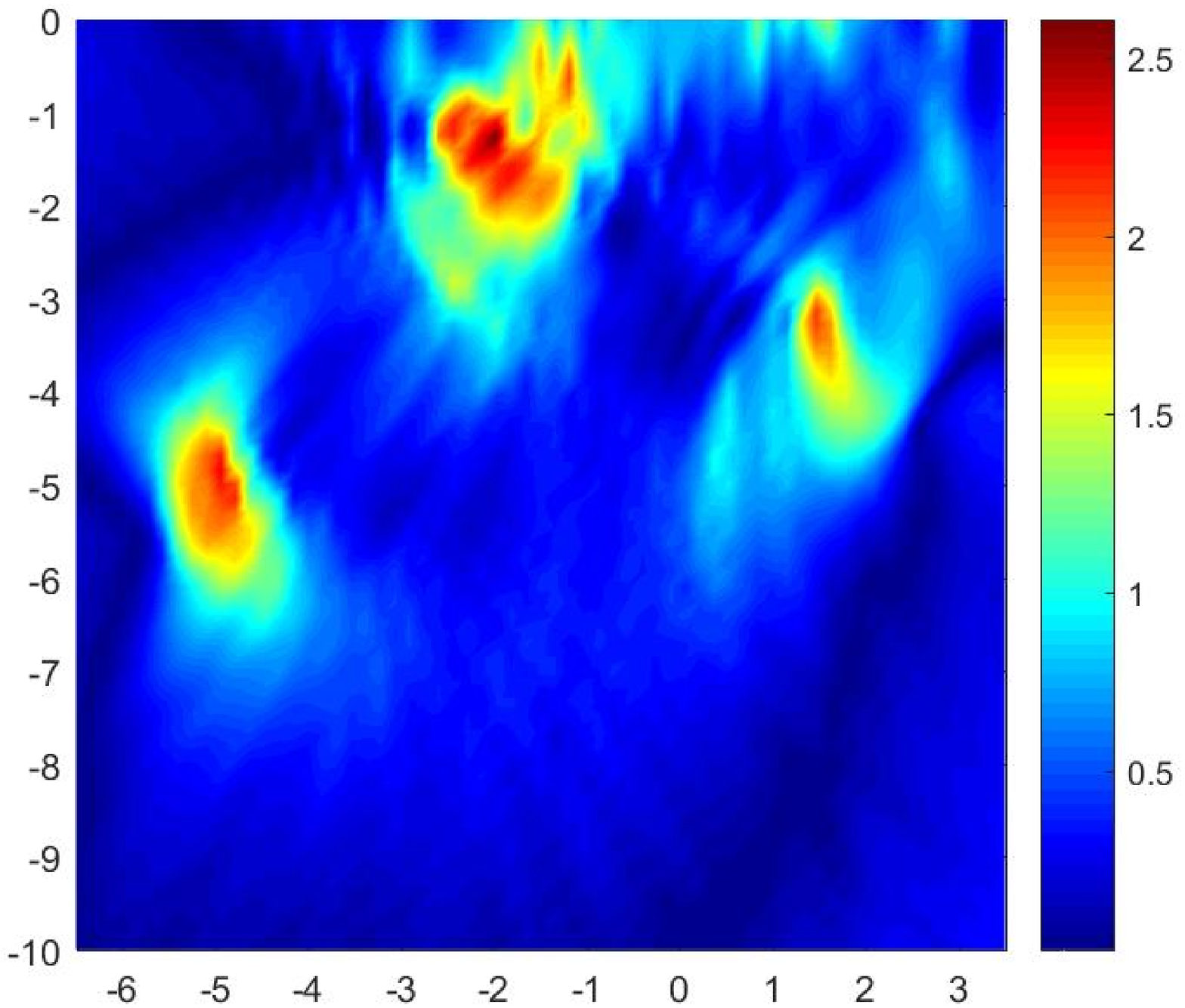}
%\caption{fig1}
\end{minipage}%
}%
\hfill
\subfigure[]
{
\begin{minipage}[t]{0.28\textwidth}
\centering
\includegraphics[width=\textwidth]{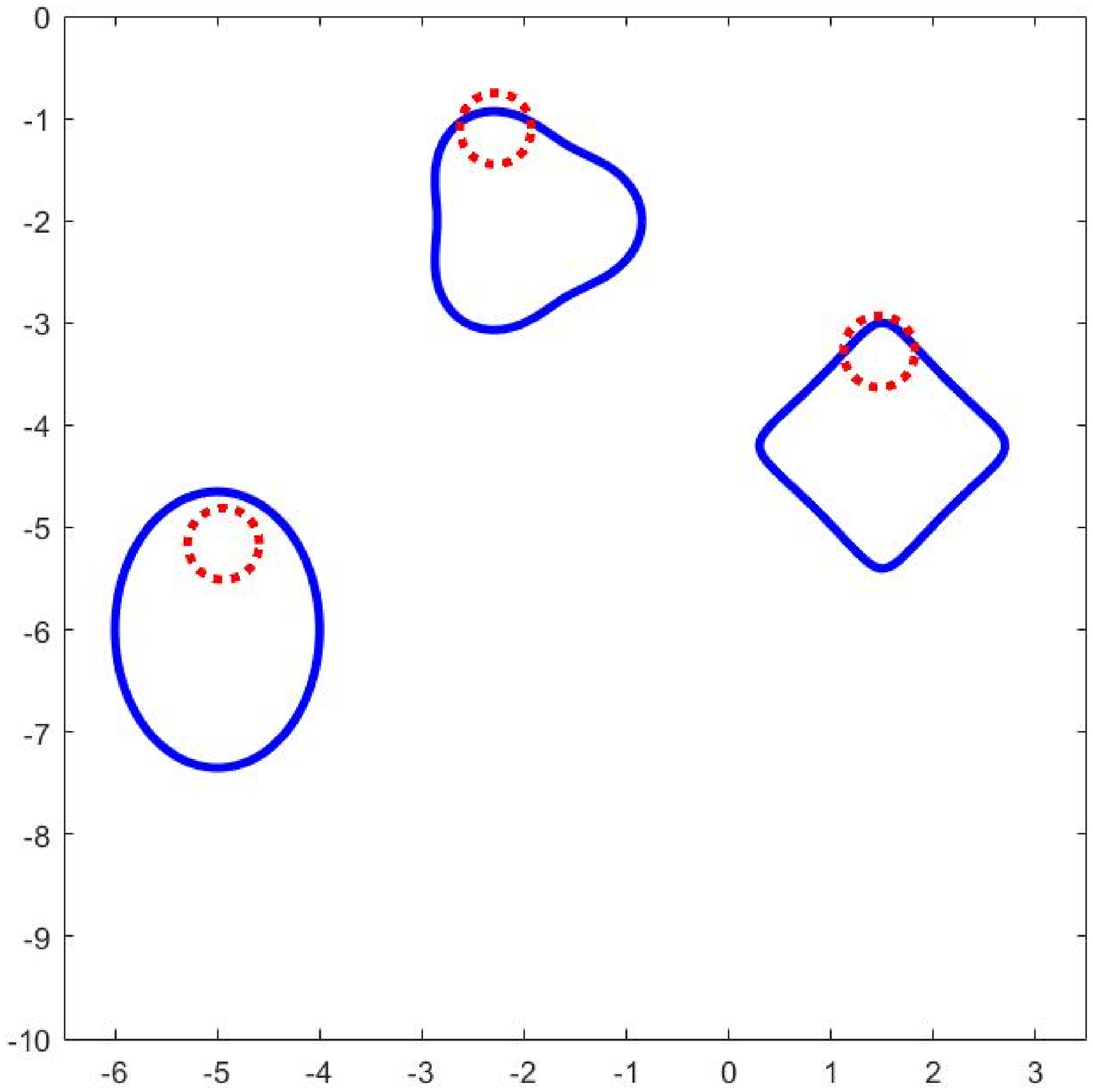}
%\caption{fig1}
\end{minipage}%
}%
\hfill
\subfigure[]
{
\begin{minipage}[t]{0.28\textwidth}
\centering
\includegraphics[width=\textwidth]{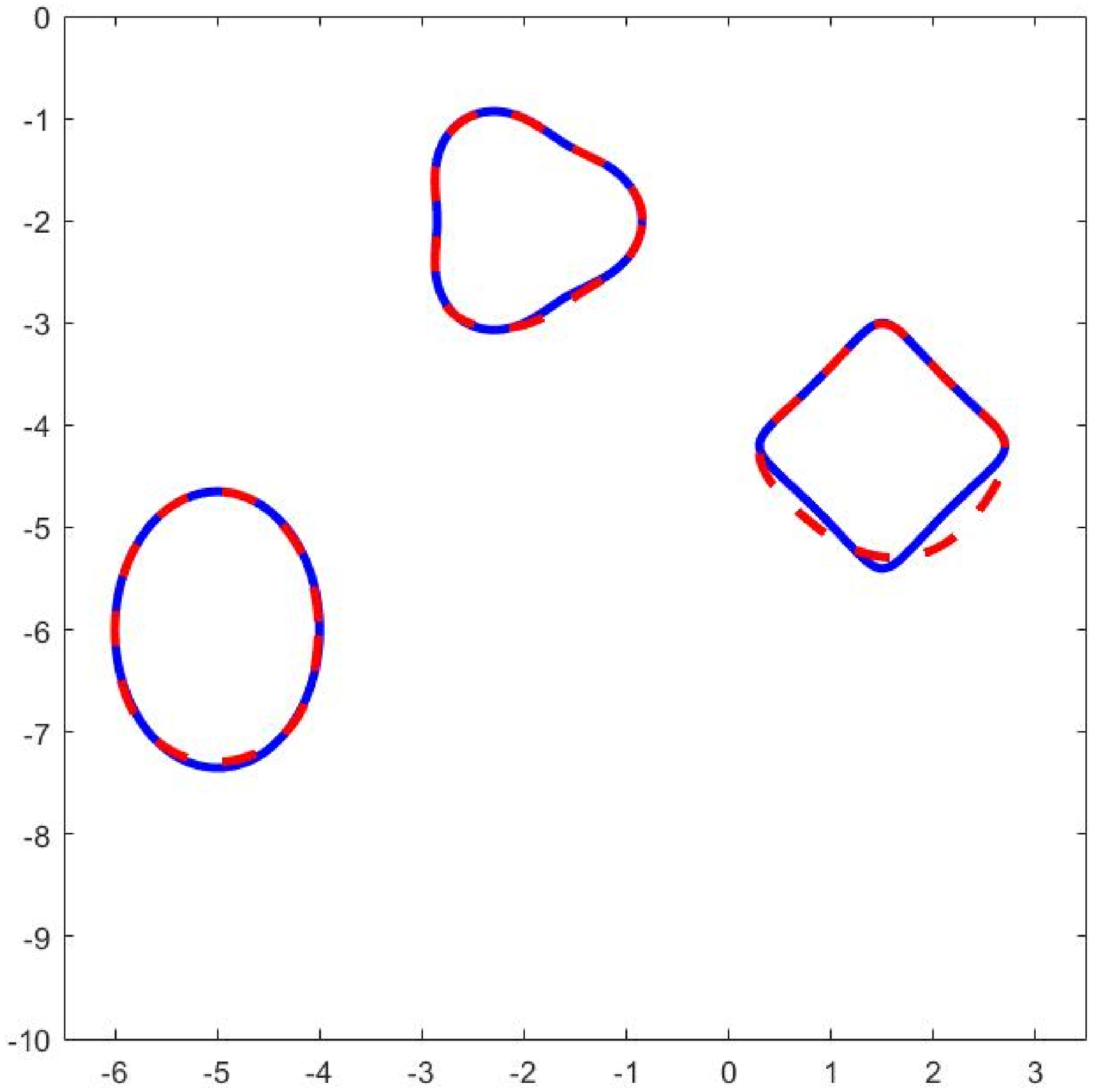}
%\caption{fig2}
\end{minipage}%
}%
\centering
\caption{Location and shape reconstruction of multiple obstacles from the phaseless far-field data
with $4\%$ noise in the case $k_+>k_-$: (a) The reconstruction result by Algorithm \ref{A1} at $k^{(1)}_+=30$ and
$k^{(1)}_{-}=k^{(1)}_{+}/2$, (b) The initial curve for Algorithm \ref{A2}, (c) The reconstructed obstacle
by Algorithm \ref{A2}  at $k^{(2)}_+=30$ and $k^{(2)}_{-}=k^{(2)}_{+}/2$.}\label{fig6}
\end{figure}

%\vspace{0.5em}
\textbf{Example 7: Reconstruction of multiple obstacles in the case $k_{+}<k_{-}$.}
%
%\vspace{0.3em}
Consider the inverse problem for reconstructing multiple obstacles in the case $k_{+}<k_{-}$,
where the obstacles are the same as in Example 6. Figure \ref{fig7}(a) presents the imaging result by the direct imaging method with $k^{(1)}_{+}=13$,
whose local maximums are at $(-2.28,-1.25),(-4.94,-5.16),(1.47,-3.13)$, respectively.
Figures \ref{fig7}(b) and \ref{fig7}(c) present the initial curve and the reconstruction result at $k^{(2)}_{+}=13$,
where the solid line represents the exact curve.
From Figure \ref{fig7}(c) it is concluded that the upper part of the shape for all obstacles can be
satisfactorily reconstructed, which is consistent with the results in Example 5.
Further, both the location and shape of the rounded triangle-shaped obstacle are
reconstructed very well.
However, the lower parts of the ellipse-shaped and rounded
square-shaped obstacles are not very accurately reconstructed.

\begin{figure}[htbp]
\centering
\subfigure[]
{
\begin{minipage}[t]{0.36\textwidth}
\centering
\includegraphics[width=\textwidth]{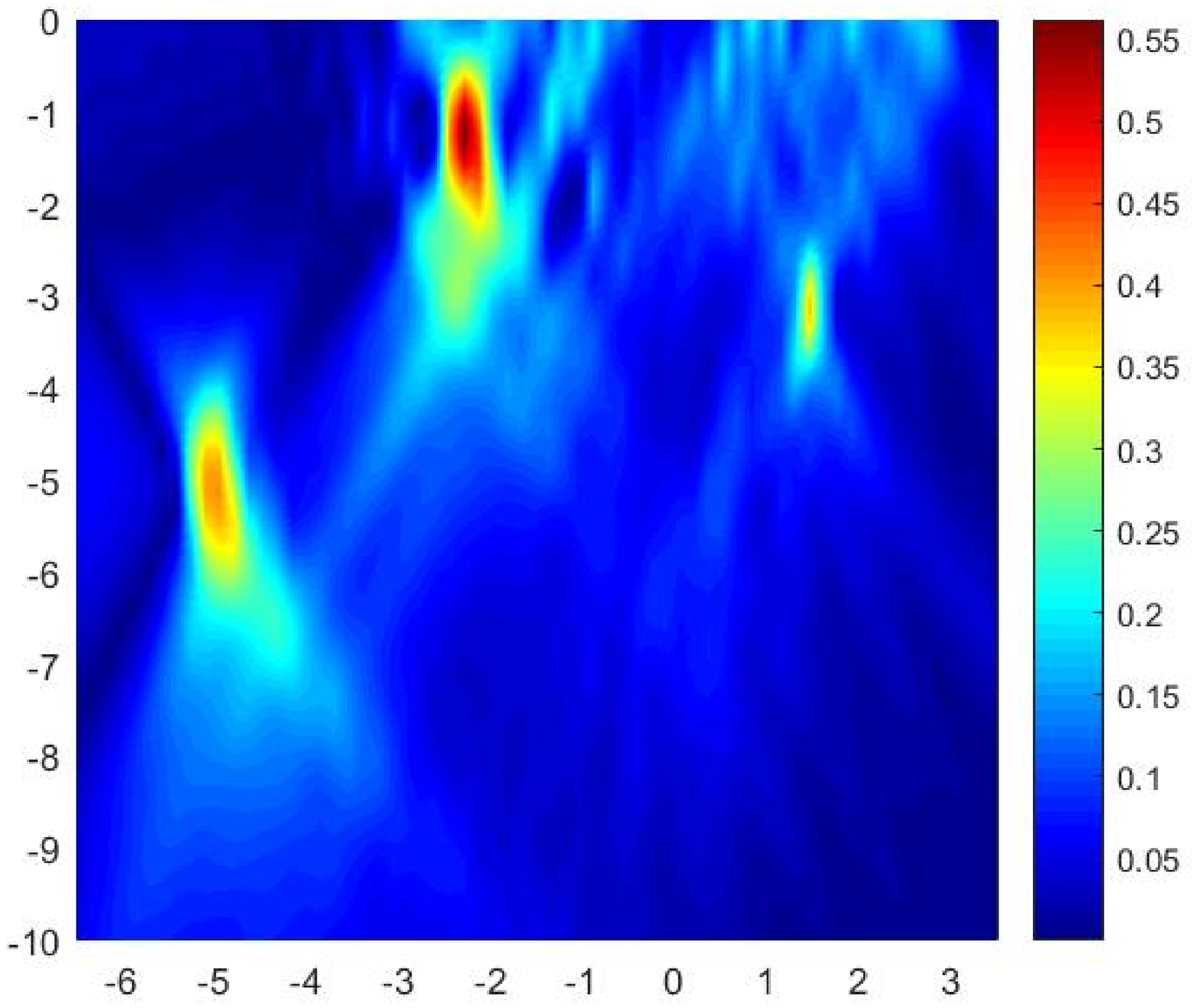}
%\caption{fig1}
\end{minipage}%
}%
\hfill
\subfigure[]
{
\begin{minipage}[t]{0.28\textwidth}
\centering
\includegraphics[width=\textwidth]{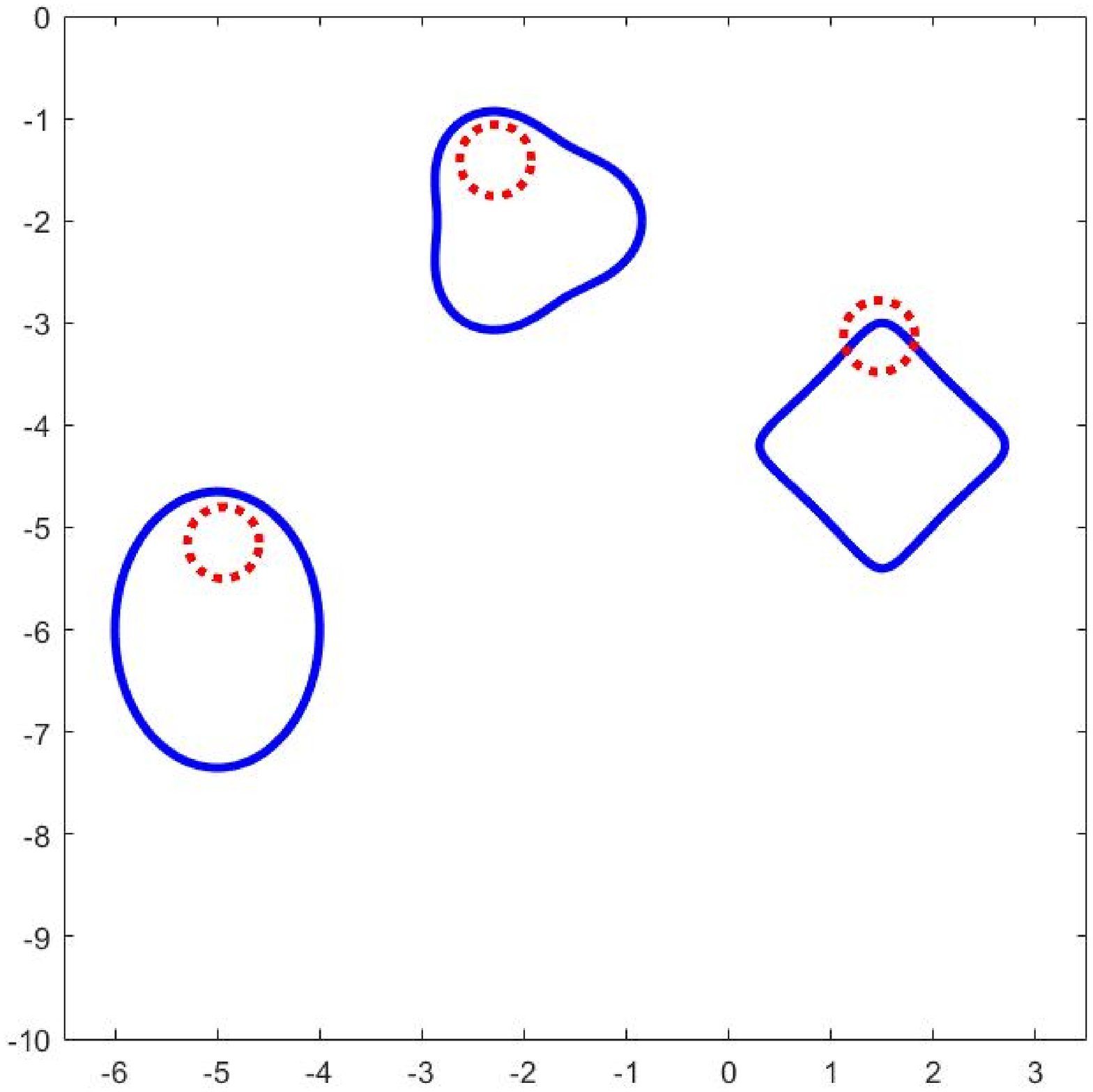}
%\caption{fig1}
\end{minipage}%
}
\hfill
\subfigure[]
{
\begin{minipage}[t]{0.28\textwidth}
\centering
\includegraphics[width=\textwidth]{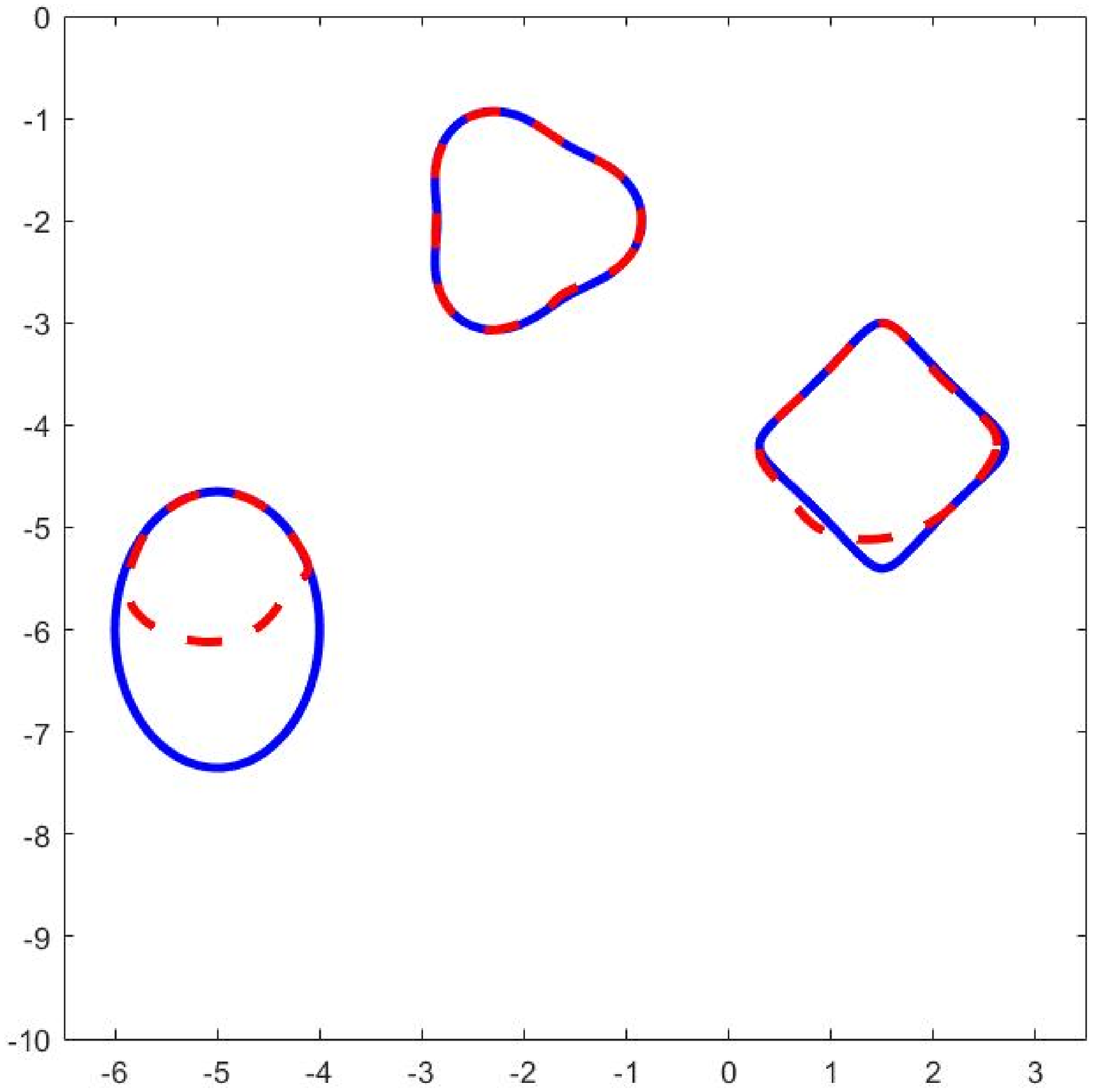}
%\caption{fig2}
\end{minipage}%
}%
\centering
\caption{Location and shape reconstruction of multiple obstacles from the phaseless far-field data
with $4\%$ noise in the case $k_+<k_-$: (a) The reconstruction result by Algorithm \ref{A1}  at $k^{(1)}_+=13$ and
$k^{(1)}_{-}=1.45k^{(1)}_{+}$, (b) The initial curve for Algorithm \ref{A2}, (c) The reconstructed obstacle
by Algorithm \ref{A2}  at $k^{(2)}_+=13$ and $k^{(2)}_{-}=1.45k^{(2)}_{+}$.}\label{fig7}
\end{figure}

\section{Conclusions and future work}{\label{sec5}}

In this paper, we have proposed two algorithms to {\color{hw}{image}} buried obstacles in the lower half-space of an unbounded
two-layered medium with only phaseless far-field data. Following the idea of \cite{BLS}, we make use of
superpositions of two plane waves as the incident fields and extend the direct imaging algorithm in \cite{RW}
and the recursive Newton-type iteration method in \cite{BLS} to the two-layered medium problem.
The direct imaging method can determine the location of the buried obstacles, providing some a priori information
for the recursive iteration method. Combining this {\color{hw}{a priori}} information with the recursive Newton-type iteration
method, the location and shape of the extended obstacles buried in the lower half-space can be recovered.

Through various numerical experiments, it has been shown that both two algorithms proposed in this paper are effective not only for the
case $k_{+}>k_{-}$ but also for the case $k_{+}<k_{-}$. However, it is observed that the reconstruction
results for the case $k_{+}> k_{-}$ are better than those for the case $k_{+}<k_{-}$.
In particular, for the case $k_{+}<k_{-}$,
the lower part of the obstacle is not very accurately reconstructed compared with
its upper part.
This may be due to the fact that the phaseless far-field data
are only measured on the upper
unit half-circle.
Therefore, certain improvements still need to be further investigated.

\section*{Acknowledgments}

The work of L. Li and J. Yang is partially supported by the NNSF of China grants {\color{hw}{11961141007}} and 61520106004, and
Microsoft Research of Asia. The work of H. Zhang is supported by the NNSF of China grant 11871466.

\end{document}